\documentclass[10pt]{article}
\usepackage{graphicx}
\usepackage{amssymb,amsmath,amsthm,amsbsy}
\usepackage{calligra,calrsfs,mathrsfs,tabularx}
\usepackage{multirow}
\usepackage{ulem,cancel}
\usepackage[usenames,dvipsnames]{color}
\usepackage{subcaption}
\usepackage{tikz}
\usetikzlibrary[patterns]
\usepackage{hyperref}
\usepackage{algorithm}
\usepackage{algpseudocode}
\usepackage{dsfont}
\usepackage{caption}
\usepackage{epstopdf}
\usepackage{epsfig}
\usepackage[numbers]{natbib}
\usepackage{pgfplots}
\usepgfplotslibrary{fillbetween}
\usetikzlibrary{patterns}
\usepackage{enumerate}
\usepackage{tikz-dimline}
\usepackage{geometry}
\numberwithin{equation}{section}

\newcolumntype{M}[1]{>{\centering\arraybackslash}m{#1}}

\newcommand{\ifcomment}{\iffalse}

\newcommand{\bs}[1]{\boldsymbol{#1}}

\newcommand{\ti}[1]{\tilde{#1}}
\newcommand{\G}{\Gamma}
\newcommand{\Om}{\Omega}

\newcommand{\cT}{{\mathcal T}}
\newcommand{\tO}{\tilde{\Omega}_h}
\newcommand{\tG}{\tilde{\Gamma}_h}
\newcommand{\tn}{{\tilde{\bs{n}}}}
\newcommand{\tns}{{\tilde{n}}}
\newcommand{\tx}{{\tilde{\bs{x}}}}
\newcommand{\tT}{{\tilde{\cT}_h}}
\newcommand{\tS}{S_{h}}

\newtheorem{definition}{Definition}
\newtheorem{thm}{Theorem}

\newtheorem{prop}{Proposition}
\newtheorem{lemma}{Lemma}
\newtheorem{rem}{Remark}
\newtheorem{assumption}{Assumption}
\newtheorem{property}{Property}
\newtheorem{corollary}{Corollary}

\newcommand{\Figures}{./Figures}

\providecommand{\keywords}[1]{\textbf{ Keywords.} #1}
\providecommand{\amsclass}[1]{\textbf{AMS Classification Index.} #1}

\title{Analysis of the Shifted Boundary Method \\ for the Poisson Problem in General Domains}
\author{Nabil~M.~Atallah \thanks{Department of Civil and Environmental Engineering, Duke University, Durham, North Carolina 27708, USA  (nabil.atallah@duke.edu) }
\and 
Claudio~Canuto\thanks{Dipartimento di Scienze Matematiche,
Politecnico di Torino,
Corso Duca degli Abruzzi 24,
I-10129 Torino, Italy (claudio.canuto@polito.it)} 
\and
Guglielmo~Scovazzi\thanks{Department of Civil and Environmental Engineering, Duke University, Durham, North Carolina 27708, USA  (guglielmo.scovazzi@duke.edu)} 
}
\begin{document}
\maketitle

%\author[duke]{Nabil M. Atallah}
%\ead{nabil.atallah@duke.edu}
%%
%\author[polito]{Claudio Canuto}
%\ead{claudio.canuto@polito.it}
%%
%\author[duke]{Guglielmo Scovazzi\corref{ca}}
%\ead{guglielmo.scovazzi@duke.edu}
%%
%\address[duke]{Department of Civil and Environmental Engineering, Duke University, Durham, North Carolina 27708, USA}
%\address[polito]{Dipartimento di Scienze Matematiche, Politecnico di Torino, Corso Duca degli Abruzzi, 24
%10129 Torino, Italy}
%\cortext[ca]{Corresponding author: Guglielmo Scovazzi}

\begin{abstract}
The shifted boundary method (SBM) is an approximate domain method for boundary value problems, in the broader class of unfitted/embedded/immersed methods. It has proven to be quite efficient in handling problems with complex geometries, ranging from Poisson to Darcy, from Navier-Stokes to elasticity and beyond.
The key feature of the SBM is a {\it shift} in the location where Dirichlet boundary conditions are applied - from the true to a surrogate boundary - and an appropriate modification (again, a {\it shift}) of the value of the boundary conditions, in order to reduce the consistency error.
In this paper we provide a sound analysis of the method in smooth and non-smooth domains, highlighting the influence of geometry and distance between exact and surrogate boundaries upon the convergence rate.
Without loss of generality, we consider the Poisson problem with Dirichlet boundary conditions as a model and we first detail a procedure to obtain the crucial shifting between the surrogate and the true boundaries. Next, we give a sufficient condition for the well-posedness and stability of the discrete problem. The behavior of the consistency error arising from shifting the boundary conditions is thoroughly analyzed, for smooth boundaries and for boundaries with corners and edges. The convergence rate is proven to be optimal in the energy norm, and is further enhanced in the $L^2$-norm.
\end{abstract}

\noindent \keywords{Embedded methods; immersed boundary methods; non-smooth domains; finite element methods; weak boundary conditions; Taylor expansions; consistency error; convergence analysis}

\noindent\amsclass{Primary 65N30; Secondary 65N12, 65N50}

%\end{frontmatter}

% main text begins ...
\section{Introduction \label{sec:intro} }

 The shifted boundary method (SBM) is an approximate domain method for boundary value problems, in the broader class of unfitted/embedded/immersed methods (see, e.g., \cite{peskin1972flow,boffi2003finite,burman2010ghost,hansbo2002unfitted,hollig2003finite,ruberg2012subdivision,burman2018cut,burman2015cutfem,kamensky2017immersogeometric,cockburn2012solving,parvizian2007finite}).
 
 In the SBM, the location where boundary conditions are applied is {\it shifted} from the true boundary to an approximate (surrogate) boundary, and, at the same time, modified ({\it shifted}) boundary conditions are applied in order to avoid a reduction of the convergence rates of the overall formulation. In fact, if the boundary conditions associated to the true domain are not appropriately modified on the surrogate domain, only first-order convergence is to be expected. The appropriate (modified) boundary conditions are then applied weakly, using a Nitsche strategy. This process yields a method which is simple, robust, accurate and efficient. 
 
 The SBM was proposed for the Poisson and Stokes flow problems in~\cite{main2018shiftedI} and was generalized in~\cite{main2018shiftedII} to the advection-diffusion and Navier-Stokes equations and in~\cite{song2018shifted} to hyperbolic conservation laws. An extension of the SBM in conjunction with high-order gradient approximations was presented in~\cite{nouveau2019high}, and the benefits of its application in the context of reduced order modeling was analyzed in~\cite{karatzas2020reduced,karatzas2019reduced,karatzas2019reduced_2}. Further rigorous mathematical analysis was pursued in~\cite{atallah2020analysis,sbm_with_claudio_darcy} for the Stokes and Darcy flow equations.
 
From a mathematical perspective, the SBM can be related to the boundary approximation method proposed by Bramble, Dupont, and Thom\`ee \cite{BrDuTh72}, although the original SBM works aimed at different, more general directions, as highlighted in the present paper.
 The purpose of this paper is to provide a sound mathematical analysis of the SBM in rather general domains in two or three dimensions. We develop a stability, consistency and convergence analysis for a model Dirichlet problem for the Poisson equation. 
The domain may have a smooth boundary, or may be a polygon in two dimensions or a polyhedron in three dimensions, possibly with curved edges/faces; no convexity assumption is made. Thus, the exact solution may not exhibit full elliptic regularity, due to the presence of corners and edges. This impacts on the accuracy of the shifted boundary conditions to be applied on the surrogate boundary, since these are obtained by performing a truncated Taylor expansion therein. The consistency of the SBM discretization, and ultimately its convergence properties, depend on the behavior of the remainder in such expansion, which is thoroughly investigated in the present article.  We analyze the separate instances when the surrogate boundary is either near a smooth portion of the true boundary or near one of its corners or edges. For each case we evaluate the rate of decay of the remainder of the Taylor expansion.

Starting from an admissible, shape-regular triangulation that may not fit and may extend beyond the physical boundary, the surrogate domain is formed by the portion of such triangulation that is contained inside the physical domain. 
In particular, classical conforming, piecewise-affine finite elements are used in the surrogate domain. 
We pose an asymptotic condition on how the surrogate boundary approaches the true boundary, which essentially requires that the distance between the surrogate and the exact boundaries tends to zero slightly faster that the meshsize (say, at least proportionally to $h^{1+\zeta}$ for some arbitrarily small $\zeta >0$, if $h$ is the meshsize). This is a mild condition, with no real impact on the choice of the triangulations to be used in practice, yet it allows us to establish the well-posedness of the SBM-based Galerkin discretization, as well as its numerical stability. (A similar assumption was already made in \cite{BrDuTh72}, although the main focus therein was on the classical case of quadratic distance from a smooth boundary, i.e., $\zeta=1$ in our notation.) Together with the consistency results discussed above, this approach allows us to prove convergence of the discretization in a suitable `energy' norm with optimal rate, then in the $L^2$-norm via a duality argument. Our analysis highlights how the rate of decay of the error depends on the shape of the exact boundary, due to the possible presence of singularities in the exact solution therein. 

It is worth mentioning that while we use the simplest model of Poisson's equation, the mathematical framework and the analysis described here can easily be extended to handle more complex models, such Darcy's, or Stokes', or elasticity equations. 

This article is organized as follows: Section~\ref{sec:def-SBM} introduces the Shifted Boundary Method for our model problem, describing in particular how the exact boundary condition is mapped to the surrogate boundary. The coercivity and continuity properties of the SBM variational form are established in Section~\ref{sec:stability}. The key Section~\ref{sec:remainder} analyzes the behavior of the Taylor remainder on the surrogate boundary. The consistency and convergence properties of our method are discussed in Section~\ref{sec:convergence} based on Strang's Second Lemma, whereas an enhanced error estimate in the $L^2$-norm is obtained in Section~\ref{sec:duality} via an Aubin-Nitsche argument. Finally, a representative numerical test is given in Sect. \ref{sec:numerical_test}.

\section{The Shifted Boundary Method \label{sec:def-SBM} }

Let $\Omega$ be a bounded and connected open region in  $\mathbb{R}^{n}$ ($n=2$ or $3$), with Lipschitz boundary $\Gamma=\partial\Omega$; let $\Gamma$ be formed by $C^2$ curves (in two dimensions) or surfaces (in three dimensions), which intersect at a finite number of vertices (in two dimensions) or edges (in three dimensions), and let $\bs{n}$ denote the outward-oriented unit normal vector to $\Gamma$. We assume that the physical problem of interest has been adimensionalized  in such a way that $\text{diam}(\Omega) \simeq 1$.

We aim at numerically solving the Dirichlet boundary value problem for the Poisson equation
\begin{equation}\label{eq:Dirichlet}
\begin{split}
-\Delta u &= f \quad \text{in \ } \Omega \,, \\
u &=g \quad \text{on \ } \Gamma \,,
\end{split}
\end{equation}
where the data $f \in L^2(\Omega)$ and $g\in H^{1/2}(\Gamma)$ satisfy additional regularity assumptions that will be made precise later on. In view of the application of the finite element method, we consider a closed domain ${\mathcal D}$ such that $\text{clos}(\Omega) \subseteq {\mathcal D}$ and we introduce a family $\cT_h$ of admissible and shape-regular triangulations of ${\mathcal D}$. Then, we restrict each triangulation by selecting those elements that are contained in $\text{clos}(\Omega)$, i.e., we form
$$
\ti{\cT}_h := \{ T \in \cT_h : T \subset \text{clos}(\Omega) \}\,.
$$ 
This identifies the {\sl surrogate domain}
$$
\tO := \text{int} (\bigcup_{T \in \ti{\cT}_h}  T ) \subseteq \Omega \,,
$$
with {\sl surrogate boundary} $\tG:=\partial \tO$ and outward-oriented unit normal vector $\ti{\bs{n}}$ to $\tG$. Obviously, $\ti{\cT}_h$ is an admissible and shape-regular triangulation of $\ti{\Omega}_h$ (see Figure~\ref{fig:SBM}). As usual, we indicate by $h_T$ the diameter of an element $T \in \ti{\cT}_h$ and by $h$ the piecewise constant function in $\tO$ such that $h_{|T}=h_T$ for all $T \in \ti{\cT}_h$.
\begin{figure}
	\centering
	\begin{tikzpicture}[scale=0.85]
	%%% fill 
	\draw [black, draw=none,name path=surr] plot coordinates { (-2,-3.4641) (-1,-1.73205) (0,-0.5) (1,1.73205) (2.6,2.2) (5,1.73205) (7,2.1) (7.62,0.8) };
	\draw [blue, name path=true] plot[smooth] coordinates {(-0.7,-3.4641) (1.75,0.75) (8.25,0.75)};
	 \tikzfillbetween[of=true and surr,split]{gray!15!};
	%%% first line of elements 
	\draw[line width = 0.25mm,densely dashed,gray] (-1,1.73205) -- (0,3.4641);
	\draw[line width = 0.25mm,densely dashed,gray] (0,3.4641) -- (2,0);
	\draw[line width = 0.25mm,densely dashed,gray] (1,1.73205) -- (1,3.4641);
	\draw[line width = 0.25mm,densely dashed,gray] (1,3.4641) -- (0,3.4641);
	\draw[line width = 0.25mm,densely dashed,gray] (1,3.4641) -- (2.6,2.2);
	\draw[line width = 0.25mm,densely dashed,gray] (1,3.4641) -- (3.25,3.4641);
	\draw[line width = 0.25mm,densely dashed,gray] (3.25,3.4641) -- (2.6,2.2);
	\draw[line width = 0.25mm,densely dashed,gray] (3.25,3.4641) -- (5,1.73205);
	\draw[line width = 0.25mm,densely dashed,gray] (3.25,3.4641) -- (6,3.4641);
	\draw[line width = 0.25mm,densely dashed,gray] (6,3.4641) -- (5,1.73205);
	\draw[line width = 0.25mm,densely dashed,gray] (6,3.4641) -- (7,2.1);
	%%% second line of elements 
	\draw[line width = 0.25mm,densely dashed,gray] (0,-0.5) -- (-2,0.5);
	\draw[line width = 0.25mm,densely dashed,gray] (-2,0.5) -- (-1,1.73205);
	\draw[line width = 0.25mm,densely dashed,gray] (-2,0.5) -- (-1,-1.73205);
	\draw[line width = 0.25mm,densely dashed,gray] (-2,0.5) -- (-2.5,-2);
	\draw[line width = 0.25mm,densely dashed,gray] (-2.5,-2) -- (-1,-1.73205);
	\draw[line width = 0.25mm,densely dashed,gray] (-2.5,-2) -- (-2,-3.4641);
	\draw[line width = 0.25mm,densely dashed,gray] (0,-0.5) -- (-1,1.73205);
	\draw[line width = 0.25mm,densely dashed,gray] (-1,1.73205) -- (1,1.73205);
	\draw[line width = 0.25mm,densely dashed,gray] (0,-0.5) -- (2,0);
	\draw[line width = 0.25mm,densely dashed,gray] (2,0) -- (1,1.73205);
	\draw[line width = 0.25mm,densely dashed,gray] (1,1.73205) -- (0,-0.5);
	\draw[line width = 0.25mm,densely dashed,gray] (2,0) -- (2.6,2.2);
	\draw[line width = 0.25mm,densely dashed,gray] (2.6,2.2) -- (1,1.73205);
	\draw[line width = 0.25mm,densely dashed,gray] (2,0) -- (4,0);
	\draw[line width = 0.25mm,densely dashed,gray] (4,0) -- (2.6,2.2);
	\draw[line width = 0.25mm,densely dashed,gray] (4,0) -- (2.6,2.2);
	\draw[line width = 0.25mm,densely dashed,gray] (2.6,2.2) -- (5,1.73205);
	\draw[line width = 0.25mm,densely dashed,gray] (5,1.73205) -- (4,0);
	\draw[line width = 0.25mm,densely dashed,gray] (4,0) -- (6,0);
	\draw[line width = 0.25mm,densely dashed,gray] (6,0) -- (5,1.73205);
	\draw[line width = 0.25mm,densely dashed,gray] (6,0) -- (7,2.1);
	\draw[line width = 0.25mm,densely dashed,gray] (7,2.1) -- (5,1.73205);
	\draw[line width = 0.25mm,densely dashed,gray] (6,0) -- (8,0);
	\draw[line width = 0.25mm,densely dashed,gray] (8,0) -- (7,2.1);
	%%% third line of elements 
	\draw[line width = 0.25mm,densely dashed,gray] (0,-0.5) -- (-1,-1.73205);
	\draw[line width = 0.25mm,densely dashed,gray] (-1,-1.73205) -- (1,-1.73205);
	\draw[line width = 0.25mm,densely dashed,gray] (2,0) -- (1,-1.73205);
	\draw[line width = 0.25mm,densely dashed,gray] (1,-1.73205) -- (0,-0.5);
	\draw[line width = 0.25mm,densely dashed,gray] (2,0) -- (3,-1.73205);
	\draw[line width = 0.25mm,densely dashed,gray] (3,-1.73205) -- (1,-1.73205);
	\draw[line width = 0.25mm,densely dashed,gray] (4,0) -- (3,-1.73205);
	\draw[line width = 0.25mm,densely dashed,gray] (2,0) -- (4,0);
	\draw[line width = 0.25mm,densely dashed,gray] (4,0) -- (3,-1.73205);
	\draw[line width = 0.25mm,densely dashed,gray] (3,-1.73205) -- (5,-1.73205);
	\draw[line width = 0.25mm,densely dashed,gray] (5,-1.73205) -- (4,0);
	\draw[line width = 0.25mm,densely dashed,gray] (4,0) -- (6,0);
	\draw[line width = 0.25mm,densely dashed,gray] (6,0) -- (5,-1.73205);
	\draw[line width = 0.25mm,densely dashed,gray] (6,0) -- (7,-1.73205);
	\draw[line width = 0.25mm,densely dashed,gray] (7,-1.73205) -- (5,-1.73205);
	\draw[line width = 0.25mm,densely dashed,gray] (6,0) -- (8,0);
	\draw[line width = 0.25mm,densely dashed,gray] (8,0) -- (7,-1.73205);
	%%% fourth line of elements 
	\draw[line width = 0.25mm,densely dashed,gray] (0,-3.4641) -- (-2,-3.4641);
	\draw[line width = 0.25mm,densely dashed,gray] (-2,-3.4641) -- (-1,-1.73205);
	\draw[line width = 0.25mm,densely dashed,gray]  (-1,-1.73205) -- (0,-3.4641);
	\draw[line width = 0.25mm,densely dashed,gray] (0,-3.4641) -- (1,-1.73205);
	\draw[line width = 0.25mm,densely dashed,gray] (0,-3.4641) -- (2,-3.4641);
	\draw[line width = 0.25mm,densely dashed,gray] (2,-3.4641) -- (1,-1.73205);
	\draw[line width = 0.25mm,densely dashed,gray] (2,-3.4641) -- (3,-1.73205);
	%%%% True boundary
	\draw [line width = 0.5mm,blue, name path=true] plot[smooth] coordinates {(-0.75,-3.681818) (1.75,0.75) (8.25,0.75)};
	%%%% Surrogate boundary
	\draw[line width = 0.5mm,red] (1,1.73205) -- (2.6,2.2);
	\draw[line width = 0.5mm,red] (2.6,2.2) -- (5,1.73205);
	\draw[line width = 0.5mm,red] (5,1.73205) --  (7,2.1);
	\draw[line width = 0.5mm,red] (1,1.73205) -- (0,-0.5);
	\draw[line width = 0.5mm,red] (0,-0.5) -- (-1,-1.73205);
	\draw[line width = 0.5mm,red] (-1,-1.73205) -- (-2,-3.4641);
	%% labels
	\node[text width=0.5cm] at (7.5,2.1) {\large${\color{red}\ti{\Gamma}_{h}}$};
	\node[text width=3cm] at (1.25,1.25) {\large${\color{red}\ti{\Omega}_{h}}$};
	\node[text width=3cm] at (0.25,-3) {\large${\color{blue}\Omega}$};
	\node[text width=0.5cm] at (8.75,0.75) {\large${\color{blue}\Gamma}$};
	\node[text width=3cm] at (4.65,1.5) {\large$\Omega \setminus \ti{\Omega}_{h} $};
	\node[text width=3cm] at (5.5,3.75) {\large$\ti{\Omega}_{h}  \subset \Omega $};
	\end{tikzpicture}
	\caption{The true domain $\Omega$, the surrogate domain $\ti{\Omega}_{h} \subset \Omega$ and their boundaries $\ti{\Gamma}_{h}$ and $\Gamma$.}
	\label{fig:SBM}
\end{figure}
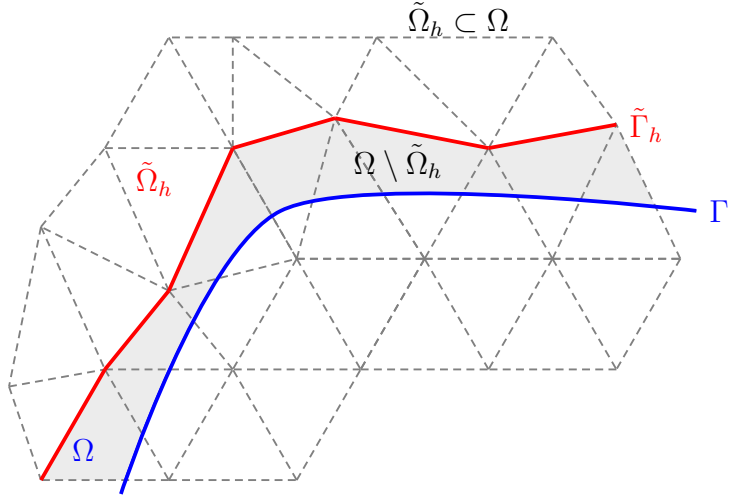
We want to discretize Problem \eqref{eq:Dirichlet} in $\tO$ rather than in $\Omega$. If by magic we knew the value of the exact solution $u$ on $\tG$, say $u_{|\tG}=\ti{g}$, we could consider the problem
\begin{equation}\label{eq:Dirichlet-h}
\begin{split}
-\Delta u &= f \quad \text{in \ } \tO \,, \\
u &=\ti{g} \quad \text{on \ } \tG \,,
\end{split}
\end{equation}
and discretize it by a Galerkin method using finite element trial and test functions on the triangulation $\tilde{\cT}_h$. In particular, we could consider Nitsche's approach \cite{nitsche1971}, which gains in flexibility by enforcing the boundary conditions in a weak manner by penalization. It is based on the following equivalent formulation of Problem \eqref{eq:Dirichlet-h}
\begin{equation}\label{eq:nitsche1}
\begin{split}
&(\nabla u, \nabla v)_{0,\tO} - (\partial_\tns u, v)_{0,\tG} - (f,v)_{0,\tO} \\ 
& \hskip 2.5cm - (u-\tilde{g}, \partial_\tns v)_{0,\tG} + \gamma\, (h^{-1} (u-\tilde{g}), v)_{0,\tG} = 0\,,
\end{split}
\end{equation}
which hold for any sufficiently smooth test function $v$ in $\tO$; here, $\partial_\tns = \tn \cdot \nabla$ denotes a normal derivative to $\tG$, whereas $\gamma >0$ is the (for the moment, arbitrary) penalization parameter. Introducing the finite dimensional subspace of $H^1(\tO)$ made of continuous, piecewise affine functions on the triangulation $\ti{\cT}_h$, 
\begin{equation}
V_h = V_h(\tO; \ti{\cT}_h) := \{ v \in H^1(\tO) : v_{|T} \in \mathbb{P}_1(T), \ \forall T \in \ti{\cT}_h \}\,,
\end{equation}
we could consider the following Galerkin discretization of Problem \eqref{eq:Dirichlet-h}: 
\begin{quote}
{\sl Find $\tilde{u}_h \in V_h$ such that $\forall v_h \in V_h$}
\end{quote}
\begin{equation}
\begin{split}
& (\nabla \tilde{u}_h, \nabla v_h)_{0,\tO} - (\partial_\tns \tilde{u}_h, v_h)_{0,\tG} -  ( \tilde{u}_h, \partial_\tns v_h)_{0,\tG} + \gamma (h^{-1} \tilde{u}_h, v_h)_{0,\tG}  \\
& \hskip 3.5cm    = (f,v_h)_{0,\tO} -  (\tilde{g}, \partial_\tns v)_{0,\tG} + \gamma \,(h^{-1} \tilde{g}, v_h)_{0,\tG}  \,.
\end{split}
\end{equation}

To make this scheme practically feasible, we need a consistent approximation of $\tilde{g}$. To obtain it, 
%the idea -- introduced in \cite{main2018shiftedI}  -- is to take a truncated Taylor expansion of $u$ at the surrogate boundary and use it  to ``transport" the Dirichlet boundary condition from $\Gamma$ to $\tG$. To this end, 
let us select a mapping
\begin{equation}\label{eq:mappingM}
\bs{M}_h : \ti{\Gamma}_h \to \Gamma\,, \qquad \ti{\bs{x}} \mapsto \bs{x}\,, 
\end{equation}
which associates to any point $\ti{\bs{x}}$ on the surrogate boundary a point $\bs{x} =\bs{M}_h(\ti{\bs{x}})$ on the true boundary. For example, $\bs{x}$ can be chosen as the closest point to $\ti{\bs{x}}$ on  $\Gamma$. We postpone to Sect. \ref{sec:def-M} the precise definition of a mapping $\bs{M}_h$ that works in practice. The mapping \eqref{eq:mappingM} can be characterized through a distance vector function $\bs{d}_{\bs{M}_h} $ defined by
\begin{equation}\label{eq:mappingd}
\bs{d}_{\bs{M}_h} (\ti{\bs{x}}) := \bs{x}  - \ti{\bs{x}}= [\bs{M}_h-\bs{I}]( \ti{\bs{x}})\,.
\end{equation}
For the sake of simplicity, we will actually set $\bs{d}=\bs{d}_{\bs{M}_h}$, so that we will write $\bs{x}=\ti{\bs{x}}+\bs{d}(\ti{\bs{x}})$. It will be useful to write $\bs{d}=\Vert \bs{d} \Vert \bs{\nu}$, where $\bs{\nu}$ is a unit vector defined on $\tG$. Note that $\bs{\nu}(\ti{\bs{x}})$ may differ from both $\ti{\bs{n}}(\ti{\bs{x}})$ and $\bs{n}(\bs{M}_h(\ti{\bs{x}}))$ (where $\bs{n}$ denotes the normal to the exact boundary $\G$).  

Assuming that $u$ is sufficiently smooth in the strip between $\tG$ and $\Gamma$, so to admit a first-order Taylor expansion pointwise, we can write
$$
g(\bs{x})=u(\bs{x}) = u(\ti{\bs{x}}+\bs{d}(\ti{\bs{x}}))=u(\tx)+ (\nabla u \cdot \bs{d})(\tx) + (R(u,\bs{d}))(\tx)\,, 
$$
where the remainder $R(u,\bs{d})$ satisfies  $|R(u,\bs{d})| = o(\Vert \bs{d}\Vert)$ as $\Vert \bs{d}\Vert \to 0$.
Equivalently, we can write  
$$
u(\tx)= g(\bs{M}_h (\tx)) -(\nabla u \cdot \bs{d})(\tx) - (R(u,\bs{d}))(\tx)\,.
$$
For the sake of simplicity, let us introduce the function on $\tG$ 
\begin{equation}\label{eq:def-tildeg}
\bar{g}(\tx):=g(\bs{M}_h (\tx))\,;
\end{equation}
then, we see that the trace $\tilde{g}$ of $u$ on $\tG$ satisfies
\begin{equation}\label{eq:trace-u}
\tilde{g} = \bar{g}-\nabla u \cdot \bs{d} - R(u,\bs{d})\,.
\end{equation}
In conclusion, on $\tG$ it holds
\begin{equation}\label{eq:u-g}
0= u-\tilde{g} \ = \ u+ \nabla u \cdot \bs{d} - \bar{g} + R(u,\bs{d}) \ = \ \tS u - \bar{g} + R_hu \,,
\end{equation}
where we have introduced the boundary operator  on $\tG$
\begin{equation}\label{eq:def-bndS}
\tS v := v+ \nabla v \cdot \bs{d} \,,
\end{equation}
and $R_h u$ is a short-hand notation for the Taylor remainder $R(u,\bs{d})$.
Neglecting the higher-order term (with respect to $\Vert \bs{d}\Vert$) in \eqref{eq:u-g} and  
recalling \eqref{eq:nitsche1}, we deduce that the exact solution $u$ satisfies the approximate equations
\begin{equation}\label{eq:nitsche2}
\begin{split}
& (\nabla u, \nabla v)_{0,\tO} - (\partial_\tns u, v)_{0,\tG} - (f,v)_{0,\tO} \\
& \hskip 2.cm  - (\tS u-\bar{g}, \partial_\tns v)_{0,\tG} + \gamma \, (h^{-1} (\tS u-\bar{g}), \tS v)_{0,\tG} \, \approx \; 0\,,
\end{split}
\end{equation}
for any sufficiently smooth test function $v$ in $\tO$. Note that in the last inner product we have replaced $v$ by $\tS v$ in order to obtain a symmetric stabilization term; on the other hand, unlike \eqref{eq:nitsche1}, the left-hand side of \eqref{eq:nitsche2} is no longer symmetric in $u$ and $v$.

Based on these approximate equations, we are led to introduce the following Galerkin discretization of Problem \eqref{eq:Dirichlet-h}: 
\begin{quote}
{\sl Find $u_h \in V_h$ such that $\forall v_h \in V_h$}
\end{quote}
\begin{equation}\label{eq:SBM-formulation-1}
\begin{split}
& \ \ \ \qquad (\nabla u_h, \nabla v_h)_{0,\tO} \!  - (\partial_\tns u_h, v_h)_{0,\tG}  \! -  ( \tS u_h, \partial_\tns v_h)_{0,\tG} +  \gamma \, (h^{-1} \tS u_h, \tS v_h)_{0,\tG} \\
& \ \ \ \ \hskip 4.5cm = (f,v_h)_{0,\tO} \! -  (\bar{g}, \partial_\tns v_h)_{0,\tG} \! + \gamma \, (h^{-1} \bar{g}, \tS v_h)_{0,\tG}\,.
\end{split}
\end{equation}

In view of the subsequent analysis, it is convenient to introduce the bilinear form
\begin{equation}\label{eq:SBM-a-form}
a_h(w,v) := (\nabla w, \nabla v)_{0,\tO} - (\partial_\tns w, v)_{0,\tG} -  ( \tS w, \partial_\tns v)_{0,\tG} + \gamma \, (h^{-1} \tS w, \tS v)_{0,\tG}
\end{equation}
and the linear form
\begin{equation}\label{eq:SBM-l-form}
\ell_h(v) := (f,v)_{0,\tO} -  (\bar{g}, \partial_\tns v)_{0,\tG} + \gamma \, (h^{-1} \bar{g}, \tS v)_{0,\tG}\,,
\end{equation}
so that the proposed SBM discretization can be written in compact form as follows:
\begin{equation}\label{eq:SBM-formulation-2}
\text{\sl Find $u_h \in V_h$ such that} \qquad a_h(u_h,v_h)=\ell_h(v_h) \quad \forall v_h \in V_h\,.
\end{equation}

Note that the form $a_h$ is well defined for all functions in $H^1(\tO)$ whose gradient admits a trace in $L^2(\tG)$. This aspect will be further investigated in Sect. \ref{sec:stability}.

\begin{rem} {\rm
The SBM \eqref{eq:SBM-formulation-1} is similar to the method proposed in \cite{BrDuTh72} for linear elements. Therein, the shift operator $\tS$ acts invariably in the direction of $\ti{\bs{n}}$ (i.e., only the choice $\bs{\nu}=\ti{\bs{n}}$ is considered). See also \cite{DuGuSc20} for a recent alternative approach.
}
\end{rem}

\subsection{A possible definition of the mapping $\bs{M}_h$}
\label{sec:def-M}

Let us assume that the boundary $\Gamma$ of the original domain $\Omega$ is partitioned into $M$ subsets $\Gamma_m$, $m=1,\dots,M$, hereafter termed {\sl sidesets}, with the following properties:
\begin{enumerate}
\item Each $\Gamma_m$ is relatively closed in $\Gamma$, and satisfies $\displaystyle{\cup_{m}} \Gamma_m = \Gamma$, $ \text{int}\,\Gamma_m \cap  \text{int}\,\Gamma_n = \emptyset$ for $m \not= n$.
\item Each  $\Gamma_m$ is ``smooth", in the sense that the normal unit vector $\bs{n}_m$ (pointing outward) exists at each $\bs{x}\in \Gamma_m$ and varies in a continuous manner.
\item The assigned Dirichlet data $g$ is a smooth function on each $\Gamma_m$.
\end{enumerate}
We aim at associating a unique sideset to each edge (in two dimensions) or face (in three dimensions) of the triangulation $\ti{\cT}_h$ that is sitting on the surrogate boundary $\tG$; all the boundary information needed to define the transported Dirichlet data $\bar{g}$ on the edge/face will be drawn from the associated sideset. \\

Consider the case of an edge (in two dimensions) $\tilde{E} \subset \tG$ (the three-dimensional case can be handled similarly) with $\tn$ as its unit normal vector.  Let $\tx_a, \, \tx_b$ be the endpoints of $\tilde{E}$ and $\bs{x}_a, \, \bs{x}_b$ be their respective closest-point projections upon $\Gamma$. Finally, let $L$ be the set of sidesets that $\bs{x}_{a}$ and $\bs{x}_{b}$ belong to. To associate a unique sideset $\Gamma_{m(\ti{E})}$ to $\ti{E}$, the following cases arise:
\begin{itemize}
\item Case 1: If $\bs{x}_{a}$ and $\bs{x}_{b}$ belong to the same unique sideset, say $\G_{1}$, then the set $L$ will have a single, element, namely $\G_{1}$, and thus $\ti{E}$ is associated with it (Figure~\ref{fig:specialEdge_Case1}) .
\item Case 2: If $\bs{x}_{a}$ and $\bs{x}_{b}$ belong to different, unique sidesets, say $\G_{1}$ and $\G_{2}$ respectively, then the set $L$ will consist of $\G_{1}$ and $\G_{2}$ (Figure~\ref{fig:specialEdge_Case2}). In such a case, we associate $\ti{E}$ with the sideset $\Gamma_m$ in $L$ such that 
\begin{align}
\Gamma_m = \arg\max_{\Gamma_s \in L} f(\Gamma_s) 
\end{align}
where $f(\Gamma_s) = \sum_{i=a,b}\ti{\bs{n}} \cdot \bs{n}_{\Gamma_s}(\bs{x}_{i})$.
% guarantees the largest value of $\sum_{i=a,b}\ti{\bs{n}} \cdot \bs{n}_{m}(\bs{x}_{i}) \, \,  \forall m \in L$ (Figure ~\ref{fig:specialEdge_Case2}).
\item Case 3: If $\bs{x}_{a}$ (or $\bs{x}_{b}$) belongs to the intersection of two sidesets, say $\G_{1}$ and $\G_{2}$, then both such sidesets are added to set $L$. At this point, we refer to Case 2 for associating $\ti{E}$ with a sideset (Figure~\ref{fig:specialEdge_Case3}).  
\end{itemize}
At last, we define the mapping $\bs{M}_h$ on $\ti{E}$ by setting
\begin{equation}\label{eq:def-M}
\bs{M}_h(\tx) := \text{the closest-point projection of $\tx$ upon $\Gamma_{m(\ti{E})}$}\,, \qquad \forall \tx \in \ti{E}\,.
\end{equation}

\begin{rem}\label{rem:def-M} {\rm 
{\sl i)} \ According to the given definition, $\bs{M}_h$ may be multi-valued at the intersection of two edges. However, this has no effect at all since $\bs{M}_h$ appears in the boundary integrals in \eqref{eq:SBM-formulation-1}, which are computed edge-wise.
{\sl ii)} \ If $\bs{M}_h(\tx)$ falls in the interior of a sideset $\Gamma_m$ and coincides with the closest-point projection of $\tx$ {\sl upon the whole $\Gamma$}, then the vector $\bs{d}$ introduced in \eqref{eq:mappingd} is aligned with $\bs{n}_m$ (Figure~\ref{fig:distance_n}). Otherwise, it may not be aligned (Figure~\ref{fig:distance_nu}). 
%\rightline{$\square$}
}
\end{rem}

\begin{figure}
	\centering
	\begin{subfigure}{0.46\textwidth}
		\centering
		\begin{tikzpicture}[scale=0.9]
		\draw[line width = 0.5mm,blue] (0,0) -- (4,0);
		\draw[line width = 0.5mm,blue] (4,0) -- (4,-4);
		\draw[->,line width = 0.25mm,-latex] (2,0) -- (2,-0.75);
		\draw[->,line width = 0.25mm,-latex] (4,-2) -- (3.25,-2);
		\node[text width=0.5cm] at (0.5,-0.5) {$\G_{1}$};
		\node[text width=0.5cm] at (1.75,-0.4) {$\bs{n}_{1}$};
		\node[text width=0.5cm] at (3.6,-3.5) {$\G_{2}$};
		\node[text width=0.5cm] at (3.75,-2.5) {$\bs{n}_{2}$};
		\draw[line width = 0.5mm,red] (1,0.25) -- (3.25,0.75);
		\node[text width=0.5cm] at (2.125,0.75) {$\ti{E}$};
		\draw[line width = 0.25mm, dashed] (1,0.25) -- (1,0);
		\draw[line width = 0.25mm, dashed] (3.25,0.75) -- (3.25,0);
		\fill (3.25,0.75) circle (0.5mm);
		\fill (1,0) circle (0.5mm);
		\fill (1,0.25) circle (0.5mm);
		\fill (3.25,0) circle (0.5mm);
		\node[text width=0.5cm] at (3.25,-0.2) {$\bs{x}_{b}$};
		\node[text width=0.5cm] at (1,-0.2) {$\bs{x}_{a}$};
		\node[text width=0.5cm] at (1,0.55) {$\ti{\bs{x}}_{a}$};
		\node[text width=0.5cm] at (3.25,1) {$\ti{\bs{x}}_{b}$};
		\end{tikzpicture}
		\caption{Case 1.}
		\label{fig:specialEdge_Case1}
	\end{subfigure}
	\qquad
	\begin{subfigure}{0.46\textwidth}
		\centering
		\begin{tikzpicture}[scale=0.9]
		\draw[line width = 0.5mm,blue] (0,0) -- (4,0);
		\draw[line width = 0.5mm,blue] (4,0) -- (4,-4);
		\draw[line width = 0.5mm,red] (3.25,1.75) -- (5,-0.5);
		\node[text width=0.5cm] at (4.3,1) {$\ti{E}$};
		\draw[->,line width = 0.25mm,-latex] (4,0.75) -- (3.4,0.25);
		\node[text width=0.5cm] at (4.1,0.35) {$\ti{\bs{n}}$};
		\draw[->,line width = 0.25mm,-latex] (2,0) -- (2,-0.75);
		\draw[->,line width = 0.25mm,-latex] (4,-2) -- (3.25,-2);
		\node[text width=0.5cm] at (0.5,-0.5) {$\G_{1}$};
		\node[text width=0.5cm] at (1.75,-0.4) {$\bs{n}_{1}$};
		\node[text width=0.5cm] at (3.6,-3.5) {$\G_{2}$};
		\node[text width=0.5cm] at (3.75,-2.5) {$\bs{n}_{2}$};
		\fill (3.25,1.75) circle (0.5mm);
		\draw[line width = 0.25mm, dashed] (3.25,1.75) -- (3.25,0);
		\node[text width=0.5cm] at (3.0,1.75) {$\ti{\bs{x}}_{a}$};
		\node[text width=0.5cm] at (3.2,-0.2) {$\bs{x}_{a}$};
		\fill (3.25,0) circle (0.5mm);
		\fill (5,-0.5) circle (0.5mm);
		\node[text width=0.5cm] at (5.25,-0.75) {$\ti{\bs{x}}_{b}$};
		\node[text width=0.5cm] at (3.8,-0.75) {$\bs{x}_{b}$};
		\fill (4,-0.5) circle (0.5mm);
		\draw[line width = 0.25mm, dashed] (5,-0.5) -- (4,-0.5);
		\end{tikzpicture}
		\caption{Case 2.}
		\label{fig:specialEdge_Case2}
	\end{subfigure}
\\
	\begin{subfigure}{0.27\textwidth}
		\centering
		\begin{tikzpicture}[scale=0.9]
		\draw[line width = 0.5mm,blue] (0,0) -- (4,0);
		\draw[line width = 0.5mm,blue] (4,0) -- (4,-4);
		\draw[line width = 0.5mm,red] (2.75,1.5) -- (5,0);
		\node[text width=0.5cm] at (4.3,1) {$\ti{E}$};
		\draw[->,line width = 0.25mm,-latex] (3.75,0.85) -- (3.4,0.275);
		\node[text width=0.5cm] at (3.9,0.45) {$\ti{\bs{n}}$};
		\draw[->,line width = 0.25mm,-latex] (2,0) -- (2,-0.75);
		\draw[->,line width = 0.25mm,-latex] (4,-2) -- (3.25,-2);
		\node[text width=0.5cm] at (0.5,-0.5) {$\G_{1}$};
		\node[text width=0.5cm] at (1.75,-0.4) {$\bs{n}_{1}$};
		\node[text width=0.5cm] at (3.6,-3.5) {$\G_{2}$};
		\node[text width=0.5cm] at (3.75,-2.5) {$\bs{n}_{2}$};
		\fill (2.75,0) circle (0.5mm);
		\draw[line width = 0.25mm, dashed] (2.75,1.5) -- (2.75,0);
		\node[text width=0.5cm] at (3.1,1.75) {$\ti{\bs{x}}_{a}$};
		\node[text width=0.5cm] at (2.6,-0.2) {$\bs{x}_{a}$};
		\fill (2.75,1.5) circle (0.5mm);
		\fill (5,0) circle (0.5mm);
		\node[text width=0.5cm] at (5.4,0) {$\ti{\bs{x}}_{b}$};
		\node[text width=0.5cm] at (4.35,-0.25) {$\bs{x}_{b}$};
		\fill (4,0) circle (0.5mm);
		\draw[line width = 0.25mm, dashed] (4,0) -- (5,0);
		\end{tikzpicture}
		\caption{Case 3.}
		\label{fig:specialEdge_Case3}
	\end{subfigure}
	\caption{Cases involved in the strategy for assigning a sideset to the surrogate edge $\ti{E}$.}
\end{figure}
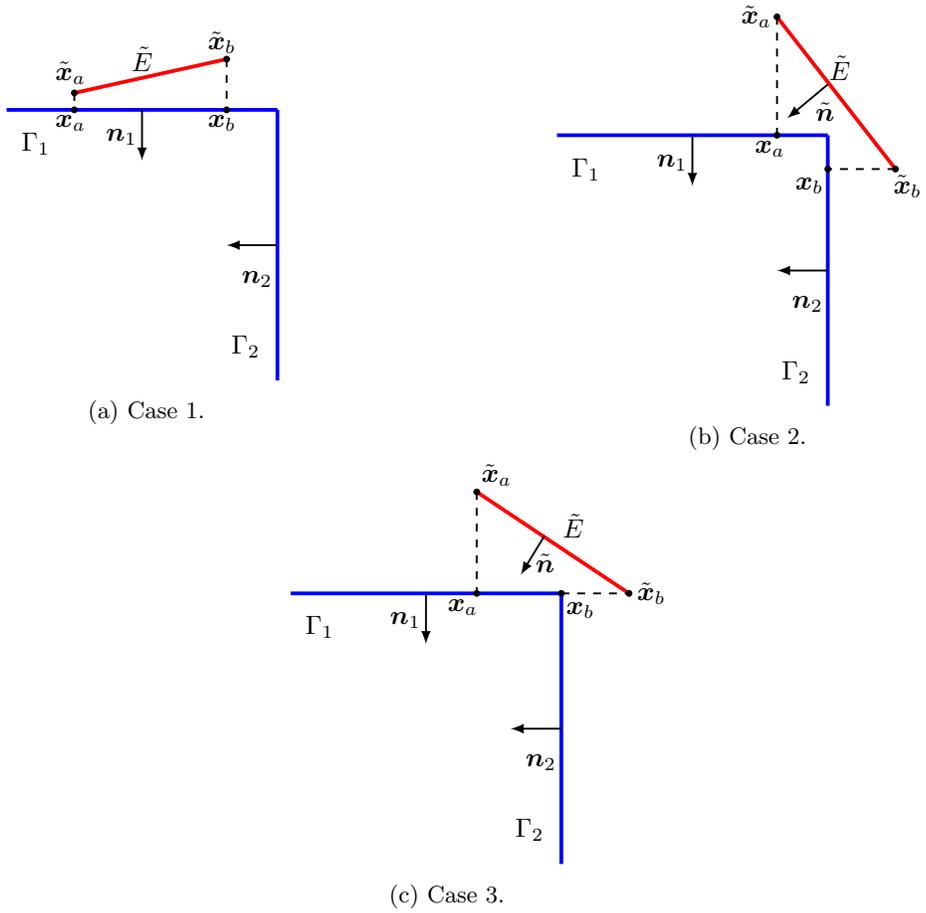
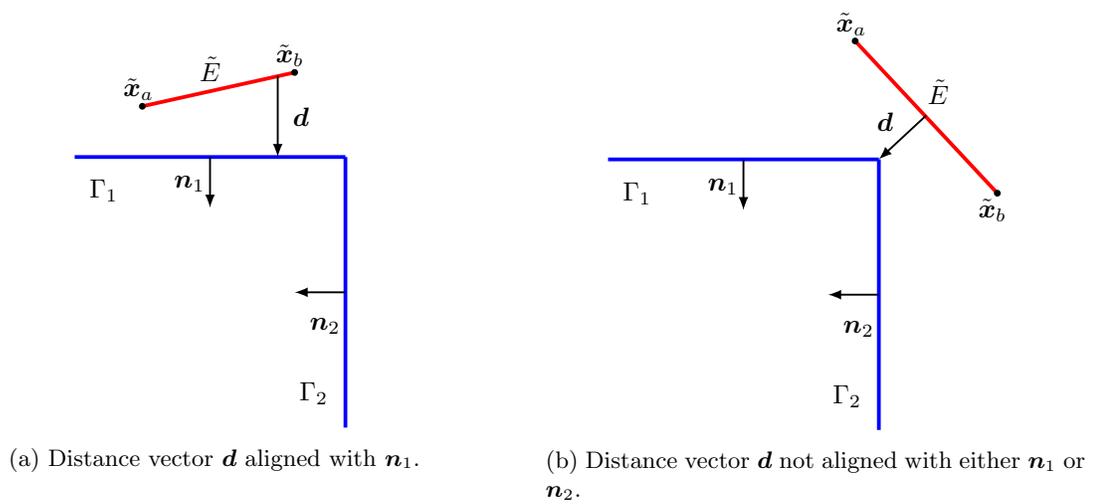
\begin{figure}
	\centering
	\begin{subfigure}{0.46\textwidth}
	\centering
	\begin{tikzpicture}[scale=0.9]
	\draw[line width = 0.5mm,blue] (0,0) -- (4,0);
	\draw[line width = 0.5mm,blue] (4,0) -- (4,-4);
	\draw[->,line width = 0.25mm,-latex] (2,0) -- (2,-0.75);
	\draw[->,line width = 0.25mm,-latex] (4,-2) -- (3.25,-2);
	\node[text width=0.5cm] at (0.5,-0.5) {$\G_{1}$};
	\node[text width=0.5cm] at (1.75,-0.4) {$\bs{n}_{1}$};
	\node[text width=0.5cm] at (3.6,-3.5) {$\G_{2}$};
	\node[text width=0.5cm] at (3.75,-2.5) {$\bs{n}_{2}$};
	\draw[line width = 0.5mm,red] (1,0.75) -- (3.25,1.25);
	\node[text width=0.5cm] at (2.125,1.3) {$\ti{E}$};
	\fill (3.25,1.25) circle (0.5mm);
	\fill (1,0.75) circle (0.5mm);
	\node[text width=0.5cm] at (1,1) {$\ti{\bs{x}}_{a}$};
	\node[text width=0.5cm] at (3.25,1.5) {$\ti{\bs{x}}_{b}$};
	\draw[->,line width = 0.25mm,-latex] (3,1.2) -- (3,0);
	\node[text width=0.5cm] at (3.5,0.6) {$\bs{d}$};
	\end{tikzpicture}
	\caption{Distance vector $\bs{d}$ aligned with $\bs{n}_{1}$.}
	\label{fig:distance_n}
\end{subfigure}
	\qquad
	\begin{subfigure}{0.46\textwidth}
	\centering
	\begin{tikzpicture}[scale=0.9]
	\draw[line width = 0.5mm,blue] (0,0) -- (4,0);
	\draw[line width = 0.5mm,blue] (4,0) -- (4,-4);
	\draw[line width = 0.5mm,red] (3.65,1.75) -- (5.75,-0.5);
	\node[text width=0.5cm] at (5,1) {$\ti{E}$};
	\draw[->,line width = 0.25mm,-latex] (4.7,0.65) -- (4,0);
    \node[text width=0.5cm] at (4.25,0.55) {$\bs{d}$};
	\draw[->,line width = 0.25mm,-latex] (2,0) -- (2,-0.75);
	\draw[->,line width = 0.25mm,-latex] (4,-2) -- (3.25,-2);
	\node[text width=0.5cm] at (0.5,-0.5) {$\G_{1}$};
	\node[text width=0.5cm] at (1.75,-0.4) {$\bs{n}_{1}$};
	\node[text width=0.5cm] at (3.6,-3.5) {$\G_{2}$};
	\node[text width=0.5cm] at (3.75,-2.5) {$\bs{n}_{2}$};
	\fill (3.65,1.75) circle (0.5mm);
	\node[text width=0.5cm] at (3.65,2) {$\ti{\bs{x}}_{a}$};
	\fill (5.75,-0.5) circle (0.5mm);
	\node[text width=0.5cm] at (5.75,-0.75) {$\ti{\bs{x}}_{b}$};
	\end{tikzpicture}
	\caption{Distance vector $\bs{d}$ not aligned with either $\bs{n}_{1}$ or $\bs{n}_{2}$.}
	\label{fig:distance_nu}
\end{subfigure}
%	\caption{Situation where the distance vector $\bs{d}$ is aligned with the normal vector $\bs{n}$ to $\Gamma$ (left) and another situation where it is not (right).
        \caption{The distance vector $\bs{d}$ and the normal vector $\bs{n}$ to $\Gamma$}
\end{figure}

\section{Coercivity and continuity analysis \label{sec:stability} }

As a first result, we establish coercivity and continuity properties of the bilinear form $a_h$. To this end, we make the following assumption, all the subsequent analysis will rely on.
\begin{assumption}\label{ass:d-smallness}
There exist constants $c_d > 0$ and $\zeta >0$ such that
\begin{equation}\label{eq:d-smallness}
\Vert \bs{d}(\tx) \Vert \leq c_d \, h_T^{1+\zeta} \,, \qquad \forall \tx \in T \cap \tG\,, \quad \forall T \in \tT \,.
\end{equation}
\end{assumption}
\noindent This condition requires the distance $\Vert \bs{d} \Vert$ between $\tG$ and $\Gamma$ to go to zero slightly faster that the local meshsize. It can be realized in practice by slightly shifting the nodes on $\tG$ towards $\Gamma$ while refining the grid.

We introduce the following mesh parameters
\begin{equation}\label{eq:mesh-param}
h_\Gamma := \max_{T\in \tT : T {\color{violet}\cap} \tG \not = \emptyset} h_T\,, \qquad h_\Omega := \max_{T\in \tT} h_T \,.
\end{equation}

The following well-known scaled trace inequalities in shape-regular triangulations will be used in different forms throughout the paper.
\begin{property}
There exists a constant $C_I >0$ independent of the meshsize such that for any $T \in \tT$ and any edge/face $E \subset \partial T$ it holds
\begin{equation}\label{eq:scaled-1}
\Vert h_T^{1/2} \nabla w \Vert_{0,E}^2 \leq C_I \, \Vert \nabla w \Vert_{0,T}^2 \,, \qquad \forall w \in \mathbb{P}_1(T)\,.
\end{equation}
This immediately gives
\begin{equation}\label{eq:scaled-2}
\Vert h^{1/2} \nabla w \Vert_{0,\tG}^2 \leq C_I \, \Vert \nabla w \Vert_{0,\tO}^2 \,, \qquad \forall w \in V_h\,.
\end{equation}
\end{property} 

We exploit Assumption \ref{ass:d-smallness} to prove the uniform coercivity of the form $a_h$ for sufficiently refined grids; see \cite[Lemma 6]{BrDuTh72} for a similar result.
\begin{prop}\label{prop:coercvity-a}
Assume that the Nitsche penalization parameter $\gamma$ satisfies $\gamma > 2C_I$, and the mesh parameter $h_\Gamma$ satisfies $h_\Gamma^\zeta \leq \frac1{4c_dC_I}$. Then, there exists a constant $\alpha>0$ independent of the meshsize, such that
\begin{equation}\label{eq:coercivity-1}
a_h(v_h,v_h) \geq \alpha \, \Vert v_h \Vert_a^2 \,, \qquad \forall v_h \in V_h\,,
\end{equation}
where
\begin{equation}\label{eq:coercivity-2}
\Vert v \Vert_a^2 := \Vert \nabla v \Vert_{0,\tO}^2+ \Vert h^{-1/2} \tS v \Vert_{0,\tG}^2\,.
\end{equation}
\end{prop} 
\begin{proof}
One has
\begin{eqnarray*}
&&\!\!\!\!\!\!\!\!\!\!\!\!   a_h(v_h,v_h) = \Vert \nabla v_h \Vert_{0,\tO}^2 \!\!\!- 2(\tS v_h, \partial_\tns v_h)_{0,\tG} \!\!+ (\nabla v_h \cdot \bs{d}, \partial_\tns v_h)_{0,\tG} \!\!+ \gamma \, \Vert h^{-1/2} \tS v_h \Vert_{0,\tG}^2  \\
&& \hskip0.9cm = \Vert \nabla v_h \Vert_{0,\tO}^2 - 2(h^{-1/2}\tS v_h, h^{1/2}\partial_\tns v_h)_{0,\tG} \\
&& \hskip 3cm + \ (h^{-1} \Vert \bs{d}\Vert  \, h^{1/2}  \partial_{\nu} v_h, h^{1/2} \partial_\tns v_h)_{0,\tG} + \gamma \, \Vert h^{-1/2} \tS v_h \Vert_{0,\tG}^2 \,,
\end{eqnarray*}
where we have used $\bs{d}=\Vert \bs{d}\Vert\, \bs{\nu}$. Recalling \eqref{eq:d-smallness} and applying \eqref{eq:scaled-2}, we get
\begin{equation*}
\begin{split}
a_h(v_h,v_h) &\geq \Vert \nabla v_h \Vert_{0,\tO}^2 - 2 C_I^{1/2} \Vert h^{-1/2} \tS v_h \Vert_{0,\tG} \Vert \nabla v_h \Vert_{0,\tO} \\
& \hskip 2.1cm - c_d C_I h_{\G}^\zeta \Vert \nabla v_h \Vert_{0,\tO}^2 + \gamma \, \Vert h^{-1/2} \tS v_h \Vert_{0,\tG}^2 \,,
\end{split}
\end{equation*}
whence, by Young inequality,
$$
a_h(v_h,v_h) \geq \left(1- \epsilon C_I - c_d C_I h_{\G}^\zeta \right) \, \Vert \nabla v_h \Vert_{0,\tO}^2  + \left(\gamma - \tfrac1\epsilon \right)\, \Vert h^{-1/2} \tS v_h \Vert_{0,\tG}^2 \,.
$$
Choosing $\epsilon = \frac1{2C_I}$, we obtain the result with $\alpha=\min \left( \tfrac14, \gamma - 2C_I \right)$.  \end{proof}

A first consequence of Proposition \ref{prop:coercvity-a} is the well-posedness and numerical stability of the SBM discretization \eqref{eq:SBM-formulation-1} with respect to the norm $\Vert u_h \Vert_a$, which uniformly dominates the $H^1$-norm of $u_h$ (see Remark \ref{rem:H1-control} below).

We now turn to discuss the continuity of the form $a_h$. In view of the convergence analysis, we consider $a_h$ as defined on a product space $V(\tO;\tT) \times V(\tO;\tT)$, where $V(\tO;\tT)$ is an infinite-dimensional subspace of $H^1(\tO)$, depending upon the triangulation $\tT$ and containing $V_h$, made of functions for which the trace of the gradient on $\tG$ is well-defined and controlled in $L^2(\tG)$. For instance, $V(\tO;\tT)$ could be the space of functions with (broken) $H^2$-regularity in $\tO$. However, since the solution $u$ of our Poisson problem may exhibit singularities at the boundary $\Gamma$ due to the presence of corners or edges, we choose a space of weaker regularity, that we are going to define.

Let us introduce a `reference' element $\hat{T}$ for all our triangulations; as usual, we assume that $\hat{T}$ has unitary diameter. Let $\hat{E} \subset \partial\hat{T}$ any edge/face of $\hat{T}$. Let $s, p$ be real numbers satisfying $0 < s \leq 1$ and $1 \leq p \leq \infty$. Then, if $t =s-\frac1p$ satisfies $t >0$, the trace  theorem (see e.g. \cite{Grisvard}) guarantees that the trace operator $\tau_{\hat{E}} v = v_{|\hat{E}}$  maps $W^{s,p}(\hat{T})$ on $W^{t,p}(\hat{E})$ continuously. On the other hand, the Sobolev imbedding theorem guarantees that $W^{t,p}(\hat{E}) \subseteq L^2(\hat{E})$ with continuous inclusion, provided $s \geq \frac{n}p - \frac{n-1}2$; this condition holds true if $s > \frac1p$ with $p \geq 2$. Thus, $v_{|\hat{E}} \in L^2(\hat{E})$ if $v \in W^{s,p}(\hat{T})$, and the bound
$$
\Vert v \Vert_{0,\hat{E}}^2 \leq \hat{C} \left( \Vert v \Vert_{0,\hat{T}}^2 + | v |_{s,p, \hat{T}}^2 \right)\,,
$$
holds with a constant $\hat{C}>0$ independent of $v$. Here, $| v |_{s,p, \hat{T}}$ denotes the fractional seminorm $\left(\int_{\hat{T}\times\hat{T}} \frac{|v(\bs{x})-v(\bs{y})|^p}{|\bs{x}-\bs{y}|^{n+sp}} d\bs{x}d\bs{y} \right)^{1/p}$ when $s<1$, or the $L^p$-norm of $\nabla v$ when $s=1$. Next, consider an element $T \in \tT$ with edge/face $E$: a scaling argument yields after a simple computation the existence of a constant $C>0$ independent of $h$ such that
\begin{equation}\label{eq:scaled-3}
\Vert h_T^{1/2} v \Vert_{0,E}^2 \leq C \left( \Vert v \Vert_{0,T}^2  +  h_T^{2s+(1 - 2/p)n} |  v |_{s,p, T}^2  \right) \,, \qquad \forall v \in W^{s,p}({T})\,.
\end{equation}
%where
%\begin{equation}\label{eq:scaled-3.5}
%\sigma := s+(\tfrac12 - \tfrac1p)n\,.
%\end{equation}
We apply such bound to each component of the gradient of a function, yielding the existence of a constant, still denoted $C_I$ as in \eqref{eq:scaled-1}, such that for all $w \in H^1(T)$ with $\nabla w \in W^{s,p}({T})$ it holds
\begin{equation}\label{eq:scaled-4}
\Vert h_T^{1/2} \nabla w \Vert_{0,E}^2 \leq C_I \left( \Vert \nabla w \Vert_{0,T}^2  +  h_T^{2s+(1 - 2/p)n} |\nabla w |_{s,p, T}^2  \right)\! , \quad \end{equation}
(Note indeed that this inequality reduces to \eqref{eq:scaled-1} when $w \in \mathbb{P}_1(T)$.) Thereby, considering the strip of elements in $\tT$ with at least one edge/face on $\tG$, i.e., the subset 
\begin{equation}\label{eq:def-cTb}
\tO^b := \bigcup \{T \in \tT^b \} \qquad \text{where \ } \tT^b:= \{ T \in \tT : \text{meas}\,(\partial T \cap \tG) > 0\}\,, 
\end{equation}
we are led to the following definition of the space $V(\tO;\tT)$. 

\begin{definition}\label{def:Wh}
Let  $s, p$ be real numbers satisfying $0 < s \leq 1$ and $2 \leq p \leq\infty$ with $s > \frac1p$. Then, set
\begin{equation}\label{eq:def-Wh1}
V(\tO;\tT) = V(\tO;\tT,s,p) := \{w \in H^1(\tO) : \nabla w_{|T} \in W^{s,p}(T) \ \ \forall T \in \tT^b \}\,,
\end{equation}
equipped with semi-norm
\begin{equation}\label{eq:def-Wh2}
| h^{s_p} \, \nabla w |_{s,p, \tO^b} := \left( \sum_{T \in \tT^b} | h_T^{s_p} \, \nabla w |_{s,p, T}^p \right)^{1/p} , \qquad s_p := s+\tfrac12-\tfrac1p \,,
\end{equation}
and norm
\begin{equation}\label{eq:def-Wh3}
\Vert w \Vert_{V(\tO;\tT)}^2 := \Vert w \Vert_a^2 + | h^{s_p} \, \nabla w |_{s,p, \tO^b}^2 \,.
\end{equation}
\end{definition}

With such definitions, we use  H\"older's inequality  in \eqref{eq:scaled-4} and the property that the measure $|\tG|$ of $\tG$ is uniformly bounded by $|\G|$, to get the existence of a constant $\bar{C}_I>0$ such that
\begin{equation}\label{eq:scaled-5}
\Vert h^{1/2} \nabla w \Vert_{0,\tG}^2 \leq \bar{C}_I \left( \Vert \nabla w \Vert_{0,\tO}^2  +  | h^{s_p} \, \nabla w |_{s,p, \tO^b}^2  \right) \,, \qquad \forall w \in V(\tO;\tT) \,,
\end{equation}
which extends \eqref{eq:scaled-2}; indeed, note that $V_h \subset V(\tO;\tT)$ and $\Vert v_h \Vert_{V(\tO;\tT)} = \Vert v_h \Vert_a$ for all $v_h \in V_h$. Note as well that $W^{1+s,p}(\tO) \subset V(\tO;\tT)$.

\begin{rem}\label{rem:H1-control} {\rm 
The norm $\Vert w \Vert_{V(\tO;\tT)}$ uniformly controls from above the standard norm $\Vert w \Vert_{1,\tO}$. Indeed, the latter is equivalent to the norm
$\left( \Vert \nabla w \Vert_{0,\tO}^2 \!\! + \Vert w \Vert_{0,\tG}^2 \! \right)^{1/2}$ and one has (assuming without loss of generality that $h \leq 1$ in $\tO$)
\begin{equation*}
\begin{split}
\Vert w \Vert_{0,\tG} & \leq \Vert h^{-1/2} w \Vert_{0,\tG} \leq  \Vert h^{-1/2} \tS w \Vert_{0,\tG} + \Vert h^{-1/2} \nabla w \cdot \bs{d} \Vert_{0,\tG} \\
& \leq 
\Vert h^{-1/2} \tS w \Vert_{0,\tG} + c_d \bar{C}_I^{1/2} h_\G^\zeta \left( \Vert \nabla w \Vert_{0,\tO} + | h^s \, \nabla w |_{s,p, \tO^b}^2 \right)^{1/2},
\end{split}
\end{equation*}
whence the result.  
}
\end{rem}

\begin{prop}\label{prop:continuity-a} The bilinear form $a_h$ is defined in $V(\tO;\tT) \times V(\tO;\tT) $ and uniformly continuous therein; precisely, there exists $A>0$ independent of $h$ such that
\begin{equation}\label{eq:continuity-a}
|\, a_h(w,v)\, | \leq A \, \Vert w \Vert_{V(\tO;\tT)} \, \Vert v \Vert_{V(\tO;\tT)}  \,, \qquad \forall w, v \in V(\tO;\tT) \,.
\end{equation}
\end{prop}
\begin{proof}
One has
\begin{eqnarray*}
&& \hskip -.35cm a_h(w,v) = (\nabla w, \nabla v)_{0,\tO} - (\partial_\tns w, \tS v)_{0,\tG} + (\partial_\tns w, \nabla v \cdot \bs{d})_{0,\tG} \\
&&  \hskip 3cm -  ( \tS w, \partial_\tns v)_{0,\tG} + \gamma \, (h^{-1} \tS w, \tS v)_{0,\tG} \\
&&  \hskip 1.cm = (\nabla w, \nabla v)_{0,\tO} \!\! - (h^{1/2}\partial_\tns w, h^{-1/2} \tS v)_{0,\tG} \!\!+ (h^{1/2} \partial_\tns w, h^{-1/2} \nabla v \cdot \bs{d})_{0,\tG} \\
&& \hskip 3.cm  -  ( h^{-1/2}\tS w, h^{1/2} \partial_\tns v)_{0,\tG} + \gamma \, (h^{-1/2} \tS w, h^{-1/2}\tS v)_{0,\tG} \;.
\end{eqnarray*}
One concludes using \eqref{eq:d-smallness} and \eqref{eq:scaled-5}. \end{proof}

\section{Behavior of the Taylor remainder} \label{sec:remainder} 

This section is devoted to the analysis of the behavior of the Taylor remainder 
$R_h u = \bar{g} - \tS u$ in the expansion of $u$ on $\tG$ introduced in \eqref{eq:u-g}. In particular, we are interested in estimating a weighted $L^2$-norm of $R_h u$ along $\tG$ in terms of the mesh parameter $h_\G$.

We will first study the remainder in the neighborhood of a smooth portion of the physical boundary $\G$; next, we will move to consider the neighborhood of a singularity in $\G$ (a corner point in a polygonal domain, an edge or a vertex in a polyhedral domain). To keep technicalities at an acceptable level, we will detail our analysis for the two-dimensional situation, while just sketching arguments in the three-dimensional case.

\subsection{Analysis near a smooth portion of the boundary}\label{sec:remainder-smooth}

Let $\G_S \subseteq \G$ be a portion of the physical boundary which admits a $C^2$ parametrization. Focusing on the two-dimensional case, let $\tilde{E} \subset \tG$ be an edge of a triangle $T=T_{\tilde{E}} \in  \tT^b$ such that $\bs{M}_h(\tilde{E}) \subset \G_S$. Recalling Assumption \ref{ass:d-smallness}, it is not restrictive to suppose $\tilde{E}$ close enough to $\G$, so that $\bs{x}=\bs{M}_h(\tx)$ is the closest-point projection upon $\Gamma$ for all $\tx \in \tilde{E}$, and the vector $\bs{d}(\tx)$ introduced in \eqref{eq:mappingd} is aligned with the unit outward normal vector $\bs{n}(\bs{x})$ to $\G$.

Let $\tx_a, \, \tx_b$ be the endpoints of $\tilde{E}$, and $h_E := \Vert \tx_b - \tx_a \Vert$ its length. Introducing the unit vector $\tilde{\boldsymbol{v}} := h_{\tilde{E}}^{-1} (\tx_b - \tx_a)$, let us parametrize the points in $\tilde{E}$ by $\tx(\tau)= \tx_a + \tau \tilde{\boldsymbol{v}}$ with $0 \leq \tau \leq h_{\tilde{E}}$. Correspondingly, points in $\bs{M}_h(\tilde{E})$ are parametrized by $\bs{x}(\tau) = \bs{M}_h(\tx(\tau))$; let us set $\bs{n}(\tau):=\bs{n}(\bs{x}(\tau))$. Furthermore, let us introduce the 2D parametrization
\begin{equation}\label{eq:def-param}
\bs{x}(\tau,\sigma) = \boldsymbol{\Phi}(\tau,\sigma) := \tx(\tau) + \sigma \bs{n}(\tau)\,, \qquad  0 \leq \tau \leq h_{\tilde{E}}, \quad 0 \leq \sigma \leq d(\tau) := \Vert \bs{d}(\tx(\tau)) \Vert\,,
\end{equation}
as shown in Figure ~\ref{fig:TaylorRemainder_SmoothBoundary}.  Note that $\boldsymbol{\Phi}$ takes values in $\bar{\Omega} \setminus \tO$, and  satisfies $\boldsymbol{\Phi}(\tau, d(\tau)) = \bs{x}(\tau)$ for $0 \leq \tau \leq h_{\tilde{E}}$.
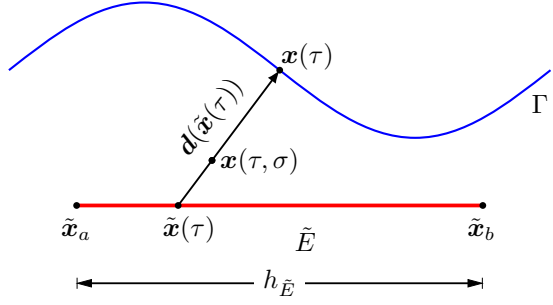
\begin{figure}
	\centering
		\begin{tikzpicture}[scale=0.9]
		\draw[line width = 0.5mm,red] (0,0) -- (6,0);
		\node[text width=0.5cm] at (3.5,-0.5) {$\ti{E}$};
		\dimline[extension start length =0,extension end length = 0,line style = {line width=0.7}]{(0,-1.15)}{(6,-1.15)}{$h_{\ti{E}}$};
		\fill (6,0) circle (0.5mm);
		\fill (0,0) circle (0.5mm);
		\fill (1.5,0) circle (0.5mm);
		\node[text width=0.5cm] at (0.05,-0.35) {$\ti{\bs{x}}_{a}$};
		\node[text width=0.5cm] at (6.05,-0.35) {$\ti{\bs{x}}_{b}$};
		\node[text width=0.5cm] at (1.55,-0.35) {$\ti{\bs{x}}(\tau)$};
        \draw[thick, blue]	(-1,2) sin (1,3) cos (3,2) sin (5,1) cos (7,2);
        \node[text width=0.5cm] at (7,1.5) {$\Gamma$};
        \draw[->,line width = 0.25mm,-latex] (1.5,0) -- (3,2);
        \fill (3,2) circle (0.5mm);
        \node[text width=0.5cm] at (3.3,2.2) {$\bs{x}(\tau)$};
        \fill (2,0.666667) circle (0.5mm);
        \node[text width=0.5cm] at (2.4,0.65) {$\bs{x}(\tau,\sigma)$};
        \fill (2,0.666667) circle (0.5mm);
   	    \node[text width=0.5cm,rotate=52.5] at (1.75,1) {$\bs{d}(\ti{\bs{x}}(\tau))$};
		\end{tikzpicture}
		\caption{Representation of a smooth portion of the boundary $\Gamma$, a surrogate edge $\ti{E}$ and their respective parametrization.}
		\label{fig:TaylorRemainder_SmoothBoundary}
\end{figure}
In order to compute the remainder $R_h u$ at the point $\tx(\tau)$ for some fixed $\tau$, let us introduce the mapping $\phi(\sigma) := u(\boldsymbol{\Phi}(\tau,\sigma))$. Then, assuming $u$ smooth enough (or applying a density argument), one has the Taylor representation
$$
\bar{g}(\tx(\tau)) = \phi(d(\tau)) = \phi(0) + \frac{{\rm d}\phi}{{\rm d}\sigma}(0) d(\tau) + \int_0^{d(\tau)} \frac{{\rm d}^2\phi}{{\rm d}\sigma^2}(s) (d(\tau)- s) \, {\rm d}s \,,
$$
with 
$$
\phi(0) = u(\tx(\tau)) \,, \quad \frac{{\rm d}\phi}{{\rm d}\sigma}(0) = \nabla u (\tx(\tau))\cdot \bs{n}(\tau)\,, \quad 
\frac{{\rm d}^2\phi}{{\rm d}\sigma^2}(s) = \bs{n}^T(\tau) {\mathcal H}u(\bs{x}(\tau,\sigma)) \bs{n}(\tau) \,,
$$
where ${\mathcal H}u$ denotes the Hessian matrix of $u$. It follows that $\phi(0) + \frac{{\rm d}\phi}{{\rm d}\sigma}(0) d(\tau) = (\tS u)(\tx(\tau))$, whence
\begin{equation}
(R_h u)(\tx(\tau)) %= \int_0^{d(\tau)} \frac{{\rm d}^2\phi}{{\rm d}\sigma^2}(s) (d(\tau)- s) \, {\rm d}s 
 = \bs{n}^T(\tau) \left(\int_0^{d(\tau)} {\mathcal H}u(\bs{x}(\tau,s)) (d(\tau)- s) \, {\rm d}s \right)\bs{n}(\tau) \,.
\end{equation}
Using H\"older's inequality for $\frac1p+\frac1q=1$ with $2 \leq p \leq \infty$, we easily get
\begin{equation*}
| (R_h u)(\tx(\tau)) | \leq (1+q)^{-1/q} (d(\tau))^{1+1/q} \left(\int_0^{d(\tau)} \Vert {\mathcal H}u(\bs{x}(\tau,s)) \Vert^p  \, {\rm d}s \right)^{1/p} \,.
\end{equation*}
Squaring and integrating over $\tilde{E}$, and using once more H\"older's inequality, we get, with $d_{\tilde{E}} := \max_{\tau \in [0,h_{\tilde{E}}]} d(\tau)$,
\begin{equation}\label{eq:bound-Rhu1}
\Vert R_h u \Vert_{0,\tilde{E}}^2 
 \leq c_p d_{\tilde{E}}^{4-2/p} h_{\tilde{E}}^{1-2/p} \left( \int_0^{h_{\tilde{E}}} \int_0^{d(\tau)} \Vert {\mathcal H}u(\bs{x}(\tau,s)) \Vert^p  \, {\rm d}s\,  {\rm d}\tau \right)^{2/p} \, .
%\nonumber \\
%&= 
%{\color{red}c_p d_{\tilde{E}}^{4-2/p} h_{\tilde{E}}^{1-2/p} \left( \int_0^{h_{\tilde{E}}} \int_0^{d(\tau)} \, \sgn \left(\boldsymbol{J\Phi}(\tau,\sigma) \right) \Vert 
%\, {\mathcal H}u(\bs{x}(\tau,s))  \Vert^p  \, \left( | {\rm det} \, \boldsymbol{J\Phi}(\tau,\sigma)| \right)^{-1}\, {\rm d}\bs{x}\right)^{2/p} 
%}
%\nonumber \\
%&= 
%{\color{red}c_p d_{\tilde{E}}^{4-2/p} h_{\tilde{E}}^{1-2/p} \left( \int_0^{h_{\tilde{E}}} \int_0^{d(\tau)} \Vert {\mathcal H}u(\bs{x}(\tau,s))  \Vert^p  \, \left( | {\rm det} \, \boldsymbol{J\Phi}(\tau,\sigma) |\right)^{-1}\, {\rm d}\bs{x}\right)^{2/p} \, ,}
\end{equation}
%{\color{red}where $\boldsymbol{J \Phi}$ is the Jacobian matrix of the mapping \eqref{eq:def-param}.} 

Next lemma allows us to apply a change of variable in the previous integral.
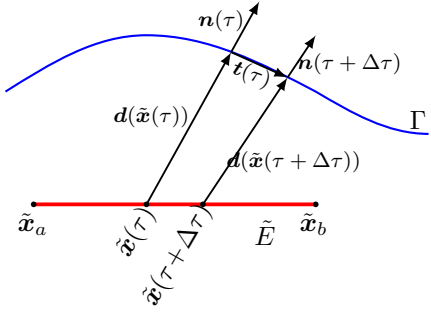
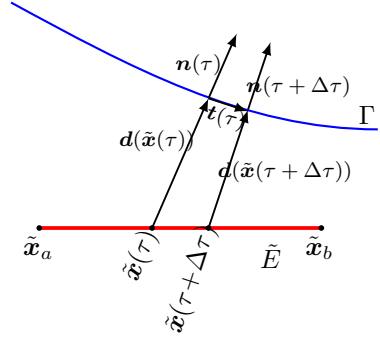
\begin{figure}
	\centering
		\begin{subfigure}{0.45\textwidth}
	\centering
	\begin{tikzpicture}[scale=0.75]
	\draw[line width = 0.5mm,red] (0,0) -- (5,0);
	\node[text width=0.5cm] at (4.25,-0.5) {$\ti{E}$};
	\fill (5,0) circle (0.5mm);
	\fill (0,0) circle (0.5mm);
	\fill (2,0) circle (0.5mm);
	\fill (3,0) circle (0.5mm);
	\node[text width=0.5cm] at (0.05,-0.35) {$\ti{\bs{x}}_{a}$};
	\node[text width=0.5cm] at (5.05,-0.35) {$\ti{\bs{x}}_{b}$};
	\node[text width=0.5cm, rotate=60] at (1.75,-0.6) {$\ti{\bs{x}}(\tau)$};
	\node[text width=1.5cm, rotate= 60] at (2.55,-0.9) {$\ti{\bs{x}}(\tau+\Delta \tau)$};
	\draw[thick, blue]	(-0.5,2) sin (2,3) cos (5,2) sin (7,1.25);
	\node[text width=0.5cm] at (7,1.5) {$\Gamma$};
	\draw[->,line width = 0.25mm,-latex] (2,0) -- (3.5,2.7);
	\draw[->,line width = 0.25mm,-latex] (3,0) -- (4.5,2.25);
	\draw[->,line width = 0.25mm,-latex] (3.5,2.7) -- (4.5,2.25);
	\draw[->,line width = 0.25mm,-latex] (3.5,2.7) -- (4,3.6);
	\draw[->,line width = 0.25mm,-latex] (4.5,2.25) -- (5,3);
	\node[text width=0.5cm] at (1.75,1.55){\footnotesize $\bs{d}(\ti{\bs{x}}(\tau))$};
	\node[text width=2cm] at (4.75,0.75){\footnotesize $\bs{d}(\ti{\bs{x}}(\tau+\Delta \tau))$};
	\node[text width=0.5cm] at (3.25,3.25){\footnotesize $\bs{n}(\tau)$};
	\node[text width=2cm] at (6,2.5){\footnotesize $\bs{n}(\tau+\Delta \tau)$};
	\node[text width=0.5cm,rotate=-20] at (3.85,2.3){\footnotesize $\bs{t}(\tau)$};
	\end{tikzpicture}
	\caption{Tangent vector for a concave boundary $\Gamma$.}
	\label{fig:TangentVector1_SmoothBoundary}
	\end{subfigure}
\qquad
%\\
%\vspace{0.5cm}
		\begin{subfigure}{0.45\textwidth}
	\centering
	\begin{tikzpicture}[scale=0.75]
	\draw[line width = 0.5mm,red] (0,0) -- (5,0);
	\node[text width=0.5cm] at (4.25,-0.5) {$\ti{E}$};
	\fill (5,0) circle (0.5mm);
	\fill (0,0) circle (0.5mm);
	\fill (2,0) circle (0.5mm);
	\fill (3,0) circle (0.5mm);
	\node[text width=0.5cm] at (0.05,-0.35) {$\ti{\bs{x}}_{a}$};
	\node[text width=0.5cm] at (5.05,-0.35) {$\ti{\bs{x}}_{b}$};
	\node[text width=0.5cm, rotate=68] at (1.75,-0.6) {$\ti{\bs{x}}(\tau)$};
	\node[text width=1.5cm, rotate= 72] at (2.7,-0.9) {$\ti{\bs{x}}(\tau+\Delta \tau)$};
	\draw[thick, blue]	(-0.5,4) sin (6,1.75);
	\node[text width=0.5cm] at (6,2) {$\Gamma$};
	\draw[->,line width = 0.25mm,-latex] (2,0) -- (3,2.3);
	\draw[->,line width = 0.25mm,-latex] (3,0) -- (3.7,2.1);
	\draw[->,line width = 0.25mm,-latex] (3,2.3) -- (3.7,2.07);
	\draw[->,line width = 0.25mm,-latex] (3,2.3) -- (3.5,3.45);
	\draw[->,line width = 0.25mm,-latex] (3.7,2.1) -- (4.1,3.3);
	\node[text width=0.5cm] at (1.75,1.55){\footnotesize $\bs{d}(\ti{\bs{x}}(\tau))$};
	\node[text width=2cm] at (4.5,1){\footnotesize $\bs{d}(\ti{\bs{x}}(\tau+\Delta \tau))$};
	\node[text width=0.5cm] at (2.75,2.9){\footnotesize $\bs{n}(\tau)$};
	\node[text width=2cm] at (5,2.5){\footnotesize $\bs{n}(\tau+\Delta \tau)$};
	\node[text width=0.5cm,rotate=-15] at (3.3,1.95){\footnotesize $\bs{t}(\tau)$};
	\end{tikzpicture}
	\caption{Tangent vector for a convex boundary $\Gamma$.}
	\label{fig:TangentVector2_SmoothBoundary}
\end{subfigure}
\caption{Graphical representation of the tangent vector for a concave (left) and convex (right) boundary $\Gamma$.}
\label{fig:TangentVector}
\end{figure}
\begin{lemma}\label{lem:detJac}
Define the region ${\mathcal R}_{\tilde{E}} := \{(\tau, \sigma) : 0 \leq \tau \leq h_{\tilde{E}}, \ 0 \leq \sigma \leq d(\tau) \}$. There exists a constant $b >0$ independent of $\tilde{E}$ such that for $h_{\tilde{E}}$ small enough the Jacobian matrix $\boldsymbol{J \Phi} $ of the mapping \eqref{eq:def-param} satisfies
$$
\left |{\rm det} \, \boldsymbol{J\Phi}(\tau,\sigma) \right |\geq b \qquad \forall (\tau,\sigma) \in {\mathcal R}_{\tilde{E}}\,.
$$
\end{lemma}
\begin{proof}  From \eqref{eq:def-param}, we get
$$
{\rm det} \, \boldsymbol{J\Phi}(\tau,\sigma) = {\rm det} \left[  \!\! \begin{array}{c} \tilde{\boldsymbol{v}} + \sigma \bs{n}'(\tau) \\ \bs{n}(\tau) \end{array} \!\! \right] =
\tilde{v}_1n_2(\tau) - \tilde{v}_2n_1(\tau) + \sigma (n_1'(\tau)n_2(\tau) - n_2'(\tau)n_1(\tau))\,.
$$
Let $\boldsymbol{t}(\tau) = (n_2(\tau), -n_1(\tau))$ be the tangent unit vector to $\Gamma$ at $\bs{x}(\tau)$, with orientation coherent with that of $\bs{x}'(\tau)$, as shown in Figure~\ref{fig:TangentVector}.
By Assumption \ref{ass:d-smallness}, we get
$$
|\tilde{v}_1n_2(\tau) - \tilde{v}_2n_1(\tau)| = |\tilde{\boldsymbol{v}} \cdot \boldsymbol{t}(\tau)| = |\tilde{\boldsymbol{n}} \cdot \boldsymbol{n}(\tau)| \geq \alpha >0
$$
for some $\alpha$ independent of $h_{\tilde{E}}$ and $\tau$, provided $h_{\tilde{E}}$ is sufficiently small. On the other hand, Frenet's formula gives
$$
\bs{n}'(\tau) = - s \, \Vert \bs{x}'(\tau) \Vert \, \kappa(\tau) \, \boldsymbol{t}(\tau)\,,
$$ 
where $\kappa (\tau)$ is the curvature of $\Gamma$ at $\bs{x}'(\tau)$, and $s=-1$ ($s=+1$, resp.) if $\Gamma$ is locally concave (convex, resp.) near $\tilde{E}$ (remember that for us $\bs{n}(\tau)$ is always pointing outward $\Omega$); see again Figure~\ref{fig:TangentVector}. Hence, the quantity
$$
|\sigma (n_1'(\tau)n_2(\tau) - n_2'(\tau)n_1(\tau))| = \sigma \,  \Vert \bs{x}'(\tau) \Vert \, |\kappa(\tau)|
$$
can be made smaller than $\frac{\alpha}2$  by choosing $h_{\tilde{E}}$ (hence, $\sigma$) sufficiently small, since both $\Vert \bs{x}'(\tau) \Vert $ and $\kappa(\tau) $ are bounded due to the assumed smoothness of $\Gamma_S$. 
\end{proof}

If we set ${\mathcal A}_{\tilde{E}} := \boldsymbol{\Phi}({\mathcal R}_{\tilde{E}}) \subset \bar{\Omega}$, Lemma \ref{lem:detJac} guarantees that for $u \in W^{2,p}({\mathcal A}_{\tilde{E}})$ it holds 
$$
 \int_0^{h_{\tilde{E}}} \int_0^{d(\tau)} \Vert {\mathcal H}u(\bs{x}(\tau,s)) \Vert^p  \, {\rm d}s\,  {\rm d}\tau \leq \frac1b \int_{{\mathcal A}_{\tilde{E}}} \Vert {\mathcal H}u(\bs{x}) \Vert^p  \, {\rm d}\bs{x}  
$$

We use this bound in \eqref{eq:bound-Rhu1}. Recalling Assumption \ref{ass:d-smallness}, we derive the existence of a constant $\bar{c}_p >0$ independent of $\tilde{E}$ such that
$$
\Vert h_{\tilde{E}}^{-1/2} R_h u \Vert_{0,\tilde{E}}^2 \leq \bar{c}_p h_{\tilde{E}}^{4 - 4/p +2\zeta'} \Vert {\mathcal H}u \Vert_{L^p({\mathcal A}_{\tilde{E}})}^2 
$$
with $\zeta'=(2-1/p)\zeta$.
Next, let us denote by $\tilde{\Gamma}_{h,S}$ the portion of $\tG$ that is mapped in $\G_S$ by $\bs{M}_h$, i.e.,
\begin{equation}\label{eq:def-GhS}
\tilde{\Gamma}_{h,S} := \{ \tilde{E} \subset \tG : \bs{M}_h(\tilde{E}) \subset \G_S \}
\end{equation}
and let $\Omega_S$ be a neighborhood of $\G_S$ in $\bar\Omega$ containing the union of all sets ${\mathcal A}_{\tilde{E}}$ for $\tilde{E} \subset \tilde{\Gamma}_{h,S}$. Then, if $u \in W^{2,p}(\Omega_S)$ (which is the case if $f \in L^p(\Omega_S)$ and $g \in W^{2-1/p,p}(\Gamma_S)$), we obtain after a new application of H\"older's inequality
\begin{equation}\label{eq:bound-Rhu2}
\begin{split}
\Vert h^{-1/2} R_h u \Vert_{0,\tilde{\Gamma}_{h,S}}^2  &%\leq \bar{c}_p \!\!\! \sum_{\tilde{E} \subset \tilde{\Gamma}_{h,S}} \!\! h_{\tilde{E}}^{4 - 4/p +2\zeta'} \Vert {\mathcal H}u \Vert_{L^p({\mathcal A}_{\tilde{E}})}^2 \  
\leq \ \bar{c}_p  \, h_{\Gamma}^{3 - 2/p +2\zeta'} \!\! \!\! \sum_{\tilde{E} \subset \tilde{\Gamma}_{h,S}} \!\! h_{\tilde{E}}^{1 - 2/p} \Vert {\mathcal H}u \Vert_{L^p({\mathcal A}_{\tilde{E}})}^2 \\
& \leq \bar{c}_p  \, h_{\Gamma}^{3 - 2/p +2\zeta'}  \big( \sum_{\tilde{E} \subset \tilde{\Gamma}_{h,S}} h_{\tilde{E}} \big)^{1-2/p} \Vert {\mathcal H}u \Vert_{L^p(\Omega_S)}^2\,.
\end{split}
\end{equation}
Since the length of $\tilde{\Gamma}_{h,S}$ can be bounded independently of $h$, we arrive at the following result.

\begin{prop}\label{prop:bound-Rh-S}
Let $\G_s \subseteq \G$ be a portion of the physical boundary which admits a $C^2$ parametrization, and let $\tilde{\Gamma}_{h,S}$ be defined in \eqref{eq:def-GhS}. Assume that in a suitable neighborhood $\Omega_S$ of $\Gamma_S$ in $\bar{\Omega}$, which contains the region between $\Gamma_S$ and $\tilde{\Gamma}_{h,S}$, the exact solution satisfies $u \in W^{2,p}(\Omega_S)$ for some $p \in [2,\infty]$. Then, there exists a constant $C_p>0$ independent of $h_\Gamma$ such that if $h_\Gamma$ is sufficiently small one has 
$$
\Vert h^{-1/2} R_h u \Vert_{0,\tilde{\Gamma}_{h,S}} \leq C_p  \, h_{\Gamma}^{3/2 - 1/p +\zeta'}  \Vert {\mathcal H}u \Vert_{L^p(\Omega_S)}^2\,,
$$
where $\zeta'=(2-1/p)\zeta >0$. 
\end{prop}

In particular, for $p=2$, neglecting the contribution from $\zeta'$, we obtain
$$
\Vert h^{-1/2} R_h u \Vert_{0,\tilde{\Gamma}_{h,S}} \leq C_2  \, h_{\Gamma} \, |u|_{2,\Omega_S} \,,
$$
whereas if $u \in W^{2,p}(\Omega_S)$ for any arbitrarily large $p$, then we can take $p=1/\zeta'$ and get
$$
\Vert h^{-1/2} R_h u \Vert_{0,\tilde{\Gamma}_{h,S}} \leq C_p  \, h_{\Gamma}^{3/2} \, |u|_{W^{2,p}(\Omega_S)} \,.
$$

At last, we briefly examine the three-dimensional case. Let $\tilde{F} \subset \tG$ be a face of a tetrahedron $T=T_{\tilde{F}} \in  \tT^b$ such that $\bs{M}_h(\tilde{F}) \subset \G_S$.
If $\tx_a, \, \tx_b, \, \tx_c$ be the vertices of $\tilde{F}$, we introduce the parametrization of $\tilde{F}$ given by $\tx(\tau_1,\tau_2)= \tx_a + \tau_1 \tilde{\boldsymbol{v}}_1 + \tau_2 \tilde{\boldsymbol{v}}_2$, where $ \tilde{\boldsymbol{v}}_i$ are unit vectors parallel to the edges of  $\tilde{F}$ meeting at $\tx_a$, and $(\tau_1,\tau_2)$ varies in a triangular region $B_{\tilde{F}}$ of the plane, with diameter $h_{\tilde{F}} \leq h_T$. Then, \eqref{eq:def-param} is replaced by
$$
\bs{x}(\tau_1,\tau_2,\sigma) = \boldsymbol{\Phi}(\tau_1,\tau_2,\sigma) := \tx(\tau_1,\tau_2) + \sigma \bs{n}(\tau_1,\tau_2)\,, 
$$
with $(\tau_1,\tau_2)\in B_{\tilde{F}}, \ 0 \leq \sigma \leq d(\tau_1,\tau_2) := \Vert \bs{d}(\tx(\tau_1,\tau_2)) \Vert$, 
which leads to the analogue of \eqref{eq:bound-Rhu1}, i.e., 
$$
\Vert R_h u \Vert_{0,\tilde{F}}^2  \leq c_p d_{\tilde{F}}^{4-2/p} h_{\tilde{F}}^{2(1-2/p)} \left( \int_{B_{\tilde{F}}} \int_0^{d(\tau_1,\tau_2)} \Vert {\mathcal H}u(\bs{x}(\tau_1,\tau_2,s)) \Vert^p  \, {\rm d}s\,  {\rm d}\tau_1\, {\rm d}\tau_2 \right)^{2/p} \,,
$$
where $d_{\tilde{F}} := \max_{(\tau_1,\tau_2)\in B_{\tilde{F}}} d(\tau_1,\tau_2)$. The analogue of Lemma \ref{lem:detJac} holds as well under Assumption \ref{ass:d-smallness}; indeed, note that 
$$
{\rm det} \, \left[ \begin{array}{c} \tilde{\boldsymbol{v}}_1 \\  \tilde{\boldsymbol{v}}_2 \\ \bs{n}(\tau_1,\tau_2) \end{array} \right] =  (\tilde{\boldsymbol{v}}_1 \wedge \tilde{\boldsymbol{v}}_2 ) \cdot \bs{n}(\tau_1,\tau_2) = \tilde{\bs{n}} \cdot \bs{n}(\tau_1,\tau_2) \,.
$$
This allows us to proceed as in the two-dimensional case, arriving at the same conclusion  as the one given in Proposition \ref{prop:bound-Rh-S}.

\subsection{Analysis near a corner in a polygonal domain} \label{sec:remainder-corner}

Assume now that $\G$ contains a corner point $C$. More precisely, assume that there exists a ball $B_c(C)$ centered at $C$ and of radius $c>0$ such that $\G_C := \G \cap B_c(C)$ is formed by two smooth branches that meet only at $C$; let $\omega \in (0,2\pi)$ be the angle formed by the two tangents to the branches at $C$, measured inside $\Omega$. To keep the forthcoming analysis simple, we actually assume that the two branches are straight lines inside $B_c(C)$; this is reasonable, since in general they approach two straight lines if the radius $c$ is taken small enough. Treating the general case would lead to results similar to those obtained below, yet with a technical burden.

It is well known \cite[Theorem 4.4.3.7]{Grisvard} that in the neighborhood $\Omega_C$ of $C$ in $\Omega$ (we can take $\Omega_C := \Omega \cap B_c(C)$ for a suitable $c$), the solution $u$
can be split into regular and singular parts, say
\begin{equation}\label{eq:decomp-u}
u = u_\text{reg} + u_\text{sing} \,.
\end{equation} 
The function $u_\text{reg}$ satisfies $-\Delta u_\text{reg} =f$ in $\Omega_C$, $u_\text{reg}=g$ on $\G_C$, and belongs to $W^{2,p}(\Omega_C)$ if we assume that $f \in L^p(\Omega)$ and $g$ is the trace on $\G_C$ of a function in $W^{2,p}(\Omega_C)$; on the other hand, $u_\text{sing}$ is harmonic in $\Omega_C$, vanishes on $\G_C$, and presents a asingularity in the radial direction while approaching $C$. By linearity, we can split the remainder as
$$
R_h u = R_h u_\text{reg} + R_h u_\text{sing} \,.
$$ 
The regular part behaves as investigated in Sect. \ref{sec:remainder-smooth}, i.e., an equivalent form of Proposition \ref{prop:bound-Rh-S} applies to $R_h u_\text{reg}$; thus, hereafter we focus on $R_h u_\text{sing}$.
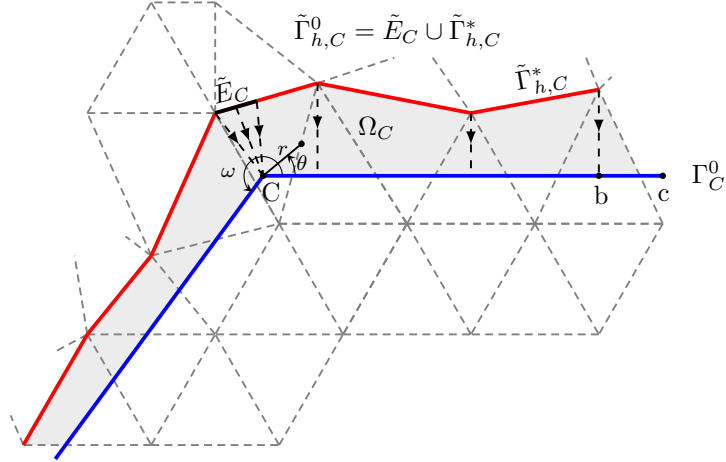
\begin{figure}
	\centering
	\begin{tikzpicture}[scale=0.85]
	%%% fill 
	\draw [draw=none,name path=surr] plot coordinates { (-2,-3.4641) (-1,-1.73205) (0,-0.5) (1,1.73205) (2.6,2.2) (5,1.73205) (7,2.1) (7.65,0.75) };
	\draw [draw=none,name path=true] plot coordinates { (-1.3,-3.4641) (1.75,0.75) (8,0.75) };
	\tikzfillbetween[of=true and surr,split]{gray!15!};
	%%% first line of elements 
	\draw[line width = 0.25mm,densely dashed,gray] (-1,1.73205) -- (0,3.4641);
	\draw[line width = 0.25mm,densely dashed,gray] (0,3.4641) -- (2,0);
	\draw[line width = 0.25mm,densely dashed,gray] (1,1.73205) -- (1,3.4641);
	\draw[line width = 0.25mm,densely dashed,gray] (1,3.4641) -- (0,3.4641);
	\draw[line width = 0.25mm,densely dashed,gray] (1,3.4641) -- (2.6,2.2);
	\draw[line width = 0.25mm,densely dashed,gray] (2.6,2.2) -- (3.8,2.6);
   \draw[line width = 0.25mm,densely dashed,gray] (4.4,2.6) -- (5,1.73205);
    \draw[line width = 0.25mm,densely dashed,gray] (7,2.1) -- (6.8,2.6);
    \draw[line width = 0.25mm,densely dashed,gray] (7,2.1) -- (7.2,2.3);
	%%% second line of elements 
	\draw[line width = 0.25mm,densely dashed,gray] (0,-0.5) -- (-1,1.73205);
	\draw[line width = 0.25mm,densely dashed,gray] (-1,1.73205) -- (1,1.73205);
	\draw[line width = 0.25mm,densely dashed,gray] (0,-0.5) -- (2,0);
	\draw[line width = 0.25mm,densely dashed,gray] (2,0) -- (1,1.73205);
	\draw[line width = 0.25mm,densely dashed,gray] (1,1.73205) -- (0,-0.5);
	\draw[line width = 0.25mm,densely dashed,gray] (2,0) -- (2.6,2.2);
	\draw[line width = 0.25mm,densely dashed,gray] (2.6,2.2) -- (1,1.73205);
	\draw[line width = 0.25mm,densely dashed,gray] (2,0) -- (4,0);
	\draw[line width = 0.25mm,densely dashed,gray] (4,0) -- (2.6,2.2);
	\draw[line width = 0.25mm,densely dashed,gray] (4,0) -- (2.6,2.2);
	\draw[line width = 0.25mm,densely dashed,gray] (2.6,2.2) -- (5,1.73205);
	\draw[line width = 0.25mm,densely dashed,gray] (5,1.73205) -- (4,0);
	\draw[line width = 0.25mm,densely dashed,gray] (4,0) -- (6,0);
	\draw[line width = 0.25mm,densely dashed,gray] (6,0) -- (5,1.73205);
	\draw[line width = 0.25mm,densely dashed,gray] (6,0) -- (7,2.1);
	\draw[line width = 0.25mm,densely dashed,gray] (7,2.1) -- (5,1.73205);
	\draw[line width = 0.25mm,densely dashed,gray] (6,0) -- (8,0);
	\draw[line width = 0.25mm,densely dashed,gray] (8,0) -- (7,2.1);
	%%% third line of elements 
	\draw[line width = 0.25mm,densely dashed,gray] (0,-0.5) -- (-1,-1.73205);
	\draw[line width = 0.25mm,densely dashed,gray] (-1,-1.73205) -- (1,-1.73205);
	\draw[line width = 0.25mm,densely dashed,gray] (2,0) -- (1,-1.73205);
	\draw[line width = 0.25mm,densely dashed,gray] (1,-1.73205) -- (0,-0.5);
	\draw[line width = 0.25mm,densely dashed,gray] (2,0) -- (3,-1.73205);
	\draw[line width = 0.25mm,densely dashed,gray] (3,-1.73205) -- (1,-1.73205);
	\draw[line width = 0.25mm,densely dashed,gray] (4,0) -- (3,-1.73205);
	\draw[line width = 0.25mm,densely dashed,gray] (2,0) -- (4,0);
	\draw[line width = 0.25mm,densely dashed,gray] (4,0) -- (3,-1.73205);
	\draw[line width = 0.25mm,densely dashed,gray] (3,-1.73205) -- (5,-1.73205);
	\draw[line width = 0.25mm,densely dashed,gray] (5,-1.73205) -- (4,0);
	\draw[line width = 0.25mm,densely dashed,gray] (4,0) -- (6,0);
	\draw[line width = 0.25mm,densely dashed,gray] (6,0) -- (5,-1.73205);
	\draw[line width = 0.25mm,densely dashed,gray] (6,0) -- (7,-1.73205);
	\draw[line width = 0.25mm,densely dashed,gray] (7,-1.73205) -- (5,-1.73205);
	\draw[line width = 0.25mm,densely dashed,gray] (6,0) -- (8,0);
	\draw[line width = 0.25mm,densely dashed,gray] (8,0) -- (7,-1.73205);
	\draw[line width = 0.25mm,densely dashed,gray] (-1,-1.73205) -- (-1.2,-0.5);
	\draw[line width = 0.25mm,densely dashed,gray] (0,-0.5) -- (-0.4,-0.2);
	\draw[line width = 0.25mm,densely dashed,gray] (-1,-1.73205) -- (-1.5,-2);
	%%% fourth line of elements 
	\draw[line width = 0.25mm,densely dashed,gray] (0,-3.4641) -- (-2,-3.4641);
	\draw[line width = 0.25mm,densely dashed,gray] (-2,-3.4641) -- (-1,-1.73205);
	\draw[line width = 0.25mm,densely dashed,gray]  (-1,-1.73205) -- (0,-3.4641);
	\draw[line width = 0.25mm,densely dashed,gray] (0,-3.4641) -- (1,-1.73205);
	\draw[line width = 0.25mm,densely dashed,gray] (0,-3.4641) -- (2,-3.4641);
	\draw[line width = 0.25mm,densely dashed,gray] (2,-3.4641) -- (1,-1.73205);
	\draw[line width = 0.25mm,densely dashed,gray] (2,-3.4641) -- (3,-1.73205);
	\draw[line width = 0.25mm,densely dashed,gray] (-2,-3.4641) -- (-2.2,-3);
	%%%% True boundary
	\draw[line width = 0.5mm,blue] (-1.5,-3.681818) -- (1.75,0.75);
	\draw[line width = 0.5mm,blue] (1.75,0.75) -- (8,0.75);
	%%%% Surrogate boundary
	\draw[line width = 0.5mm,red] (1,1.73205) -- (2.6,2.2);
	\draw[line width = 0.5mm,red] (2.6,2.2) -- (5,1.73205);
	\draw[line width = 0.5mm,red] (5,1.73205) --  (7,2.1);
	\draw[line width = 0.5mm,red] (1,1.73205) -- (0,-0.5);
	\draw[line width = 0.5mm,red] (0,-0.5) -- (-1,-1.73205);
	\draw[line width = 0.5mm,red] (-1,-1.73205) -- (-2,-3.4641);
	%%%% Surrogate maps to true
	\draw[line width = 0.5mm,black] (1,1.73205) -- (1.65,1.92);
	\draw[->,line width = 0.25mm,black,dashed,-latex] (1,1.73205) -- (1.375,1.241025);
	\draw[line width = 0.25mm,black,dashed] (1.375,1.241025) -- (1.75,0.75);
	\draw[->,line width = 0.25mm,black,dashed,-latex] (1.65,1.92) -- (1.7,1.335);
	\draw[-,line width = 0.25mm,black,dashed] (1.7,1.335) -- (1.75,0.75);
	\draw[->,line width = 0.25mm,black,dashed,-latex] (1.325,1.826025) -- (1.537,1.2880125);
	\draw[-,line width = 0.25mm,black,dashed] (1.5375,1.2880125) -- (1.75,0.75);
	\draw[->,line width = 0.25mm,black,dashed,-latex] (2.6,2.2) -- (2.6,1.475);
	\draw[-,line width = 0.25mm,black,dashed] (2.6,1.475) -- (2.6,0.75);
	\draw[->,line width = 0.25mm,black,dashed,-latex] (5,1.73205) -- (5,1.241025);
	\draw[-,line width = 0.25mm,black,dashed] (5,1.241025) -- (5,0.75);
	\draw[->,line width = 0.25mm,black,dashed,-latex] (7,2.1) -- (7,1.425);
	\draw[-,line width = 0.25mm,black,dashed] (7,1.425) -- (7,0.75);
	%% labels
	\node[text width=0.5cm,rotate=15] at (1.25,2.1) {$\ti{E}_{C}$};
	\node[text width=0.5cm,rotate=10] at (6,2.25) {$\ti{\Gamma}^{*}_{h,C}$};
	\node[text width=3cm] at (4,3) {$\ti{\Gamma}^{0}_{h,C} = \ti{E}_{C} \cup \ti{\Gamma}^{*}_{h,C}$};
	\node[text width=3cm] at (5,1.5) {$\Omega_{C}$};
	\node[text width=0.5cm] at (8.75,0.75) {$\Gamma^{0}_{C}$};
	\fill (7,0.75) circle (0.5mm);
	\node[text width=0.3cm] at (7.1,0.45) {b};
	\fill (8,0.75) circle (0.5mm);
	\node[text width=0.3cm] at (8.1,0.45) {c};
	\fill (1.75,0.75) circle (0.5mm);
	\node[text width=0.3cm] at (1.9,0.5) {C};
	\draw[black,->,-stealth] (2.05,0.75) arc (0:235:0.3);
	\node[text width=0.3cm] at (1.27,0.8) {\footnotesize $\omega$};
	\draw[-,line width = 0.25mm,black] (1.75,0.75) -- (2.35,1.25);
	\node[text width=0.3cm] at (2.15,1.15) {\footnotesize $r$};
	\fill (2.35,1.25) circle (0.5mm);
	\draw[black,->,-stealth] (2.25,0.75) arc (0:42:0.5);
	\node[text width=0.3cm] at (2.45,0.95) {\footnotesize $\theta$};
	\end{tikzpicture}
	\caption{Graphical representation of a non-smooth portion of the boundary $\Gamma$, surrogate edges $\ti{\Gamma}^{0}_{h,C}$ and the parametrization of $\Om_{C}$.}
	\label{fig:TaylorRemainder_Corner}
\end{figure}
To perform our study, it is convenient to choose a Cartesian coordinate system centered at $C$, with one of the branches of $\G_C$ sitting on the horizontal axis and the other one obtained by a counterclockwise rotation of angle $\omega$ inside $\Omega_C$. In polar coordinates $(r, \theta)$, we thus have $\Omega_C = \{(r,\theta) : 
0 < r  < c, \ 0 < \theta < \omega \}$ (see Figure~\ref{fig:TaylorRemainder_Corner}). In these coordinates, if we set $\lambda_m = \frac{m\pi}\omega > \frac12$ and $M:=\{ m \in \mathbb{N} : \lambda_m < \frac2q, \, \lambda_m \not = 1 \}$, with $\frac1p+\frac1q=1$, and we assume that $\frac{2\omega}{q\pi}$ is not an integer, then $u_\text{sing}$ has the form
\begin{equation}\label{eq:sing-u}
u_\text{sing}(r,\theta) = \sum_{m \in M} C_m \, r^{\lambda_m}\sin(\lambda_m \theta) \,,
\end{equation} 
where $C_m=C_m(f,g)$ are real coefficients depending linearly upon the data $f$ and $g$. Note that $M$ may be empty, in which case the singular part is zero. Otherwise, the behavior of $u_\text{sing}$ near the corner is dictated by the smallest value of $\lambda_m$, which we denote by $\lambda$. Thus, we are led to consider the remainder $R_h w$ in the Taylor expansion of 
\begin{equation}\label{eq:2d-singular}
w(r,\theta) := r^\lambda \sin(\lambda \, \theta) \,, \qquad \lambda > \tfrac12 \,.
\end{equation} 

Let us assume that there exists a non-empty portion $\tilde{\Gamma}_{h,C}$  of $\tG$ that is mapped into $\G_C$ by $\bs{M}_h$. Since the radius $c$ of the ball $B_c(C)$ is fixed, 
Assumption \ref{ass:d-smallness} implies that $\tilde{\Gamma}_{h,C} \subset \Omega_C$ for sufficiently small $h_\Gamma$. In this condition, we focus on the portion $\tilde{\Gamma}_{h,C}^0$ of $\tilde{\Gamma}_{h,C}$ that is mapped into the horizontal part of $\G_C$, i.e. $\G_C^0 := \{(r,\theta) : 0 \leq r \leq c,  \theta=0\}$; the remaining portion of $\tilde{\Gamma}_{h,C}$ can be handled similarly. We may assume that a (possibly empty) part of an edge $\tilde{E} \subset \tilde{\Gamma}_{h,C}^0$ is mapped by $\bs{M}_h$ to the origin $C$, whereas the complementary part in $\tilde{\Gamma}_{h,C}^0$ is mapped in $\G_C^0 \setminus \{C\}$ by the closest-point projection (see again Figure~\ref{fig:TaylorRemainder_Corner}).

Let $\tilde{E}_C$ be the part of $\tilde{E}$ which is mapped to $C$; let us introduce a linear parametrization, say $\tilde{E}_C = \{(r(\tau),\theta(\tau)), 0 \leq \tau \leq \tau_{\max } \}$, with $\tau_{\max } \leq h_E$. Therein, one has
$$
(R_h w)(r,\theta) = -(S_h w)(r,\theta) =  - \left(w(r,\theta) + \frac{\partial w}{\partial r}(r,\theta)(-r) \right) = (\lambda -1)\,  r^\lambda \sin(\lambda \, \theta) \,,
$$
whence
$$
\Vert R_h w \Vert_{0,\tilde{E}_C}^2 = (\lambda -1)^2 \int_0^{\tau_{\max }} r^{2 \lambda}(\tau) \sin^2(\lambda \, \theta(\tau))\, {\rm d}\tau \leq (\lambda -1)^2 \max_\tau   r^{2\lambda} \, \tau_{\max }.
$$
Thus, by Assumption \ref{ass:d-smallness}, we get the existence of a constant $K_1>0$ such that
\begin{equation}\label{eq:bound-Rh-corner}
\Vert h^{-1/2} R_h w \Vert_{0,\tilde{E}_C} \leq K_1 \, h_\Gamma^{\lambda(1+\zeta)}\;. 
\end{equation}

The part $\tilde{\Gamma}_{h,C}^\star$ of $\tilde{\Gamma}_{h,C}^0$ that is mapped on $\G_C^0 \setminus \{C\}$ by the closest-point projection may be parametrized as $\tilde{\Gamma}_{h,C}^\star = \{\tx = (x,\gamma(x)) : 0 < x \leq  b \}$, where $\gamma$ is a piecewise affine function and $0 < b \leq c$ with $c-b \leq h_\G$. Therein, one has, with $y=\gamma(x)$,
$$
(R_h w)(x,y) = -(S_h w)(x,y) =  - (w(x,y) + \frac{\partial w}{\partial y}(x,y)(-y) ) \,.
$$
Using $r=(x^2+y^2)^{1/2}$ and $\theta = \arctan \frac{y}{x}$ to express $w$ in Cartesian coordinates, one gets
\begin{equation}\label{eq:bound-Rh-nearC-0}
(R_h w)(x,y) = r^{\lambda -2} \left(\lambda y^2 \sin(\lambda \, \theta) +  \lambda x y \cos(\lambda \, \theta) \right) - r^\lambda \sin(\lambda \, \theta) \,.
\end{equation}
To bound the norm of $R_h w$ on $\tilde{\Gamma}_{h,C}^\star$ we observe that the distance vector $\boldsymbol{d}=\boldsymbol{x}-\tilde{\boldsymbol{x}}$ can be written on $\tilde{\Gamma}_{h,C}^\star$ as $\boldsymbol{d}(\tilde{\boldsymbol{x}})=\gamma(x) \boldsymbol{\nu}$ with $ \boldsymbol{\nu}=(0,-1)$; thus, let us introduce the quantity 
$$
d := \sup_{0 < x \leq b} \gamma(x)    = \sup_{0 < x \leq b} \Vert \bs{d}(x,\gamma(x)) \Vert  \leq  c_d h_\Gamma^{1+\zeta}
$$
Then, we split 
\begin{equation}\label{eq:bound-Rh-nearC-1}
\Vert R_h w \Vert_{0, \tilde{\Gamma}_{h,C}^\star}^2 = \int_0^b [(R_h w)(x,\gamma(x))]^2 \, {\rm d}x = \int_0^d+\int_d^b [(R_h w)(x,\gamma(x))]^2 \, {\rm d}x \;.
%& \leq K_3 \left( \int_0^b r^{2\lambda -4} \, \gamma^4(x) \sin^2(\lambda \, \theta)   \, {\rm d}x   + \int_0^b r^{2\lambda -4} \, x^2 \gamma^2(x) \, {\rm d}x 
%+ \int_0^b r^{2\lambda} \sin^2(\lambda \, \theta)   \, {\rm d}x 
\end{equation}
In the interval $0 \leq x \leq d$, we use $\gamma^2(x)\leq x^2+\gamma^2(x)=r^2$ and $\sin^2(\lambda \theta)\leq 1$, $\cos^2(\lambda \theta)\leq 1$ to get
\begin{equation}\label{eq:bound-Rh-nearC-2}
\int_0^d [(R_h w)(x,\gamma(x))]^2 \, {\rm d}x  %\lesssim  \int_0^d r^{2\lambda -4} (\gamma^4(x) + x^2\gamma^2(x) +r^4) \, {\rm d}x 
\lesssim \int_0^d  (x^2 + \gamma^2(x))^\lambda   \, {\rm d}x \lesssim d^{2\lambda+1}\,.
\end{equation}
In the interval $d \leq x \leq b$, we write \eqref{eq:bound-Rh-nearC-0} as
$$
(R_h w)(x,y) = r^{\lambda -2} \lambda y^2 \sin(\lambda \, \theta) +  r^{\lambda -2} \left(\lambda x y \cos(\lambda \, \theta) - r^2 \sin(\lambda \, \theta) \right) \;.
$$
On the one hand, using now $x^2+\gamma^2(x)\geq x^2$, we get
\begin{equation}\label{eq:bound-Rh-nearC-3}
\int_d^b r^{2\lambda -4} \gamma^4(x) \sin^2(\lambda \, \theta)   \, {\rm d}x \lesssim d^4 \int_d^b x^{2\lambda-4}  \, {\rm d}x  \lesssim d^{2\lambda+1}\,.
\end{equation}
On the other hand, observing that $0 \leq \frac{\gamma(x)}x \leq 1$ and using the Taylor expansions of $\arctan, \ \sin, \ \cos$ at the origin, we easily see that
$$
\lambda x \gamma(x) \cos(\lambda \, \theta) - (x^2+\gamma^2(x)) \sin(\lambda \, \theta) \  \simeq \ \lambda \frac{\gamma^3(x)}x  \qquad  \text{uniformly in $d \leq x\leq b$}
 $$ 
with 
\begin{equation}\label{eq:bound-Rh-nearC-4}
\int_d^b r^{2\lambda -4} \frac{\gamma^6(x)}{x^2}   \, {\rm d}x \lesssim d^6 \int_d^b x^{2\lambda-6}  \, {\rm d}x  \lesssim d^{2\lambda+1}\,.
\end{equation}

Summarizing, from \eqref{eq:bound-Rh-corner} and \eqref{eq:bound-Rh-nearC-1}-\eqref{eq:bound-Rh-nearC-4}, and from Assumption \ref{ass:d-smallness}, we easily get the existence of a constant $K_4$ independent of $h_\Gamma$ such that
\begin{equation}\label{eq:bound-Rh-nearC-5}
\Vert h^{-1/2}R_h w \Vert_{0, \tilde{\Gamma}_{h,C}^0} \leq K_4 h_\Gamma^{\lambda(1+\zeta)} \;.
\end{equation}

At last, recalling \eqref{eq:decomp-u} and \eqref{eq:sing-u}, we arrive at the following counterpart of Proposition \ref{prop:bound-Rh-S}.
\begin{prop}\label{prop:bound-Rh-C}
Suppose that $\Gamma$ has a corner point $C$ forming an angle $\omega \in (0, 2\pi)$, while away of $C$ it admits a $C^2$-parametrization. Let $\Omega_C :=\Omega \cap B_c(C)$ be a neighborhood of $C$ in $\Omega$, where $B_c(C)$ is the ball of radius $c>0$ centered at $C$ and $c$ is assumed to be sufficiently small but independent of $h$; let $\tilde{\Gamma}_{h,C} := \tG \cap B_c(C)$ be non-empty and mapped by $\bs{M}_h$ into $\Gamma\cap B_c(C)$. Define
$$ 
M:=\{ m \in \mathbb{N} : \lambda_m := \frac{m\pi}\omega < \frac2q, \, \lambda_m \not = 1 \}\,, \qquad   \tfrac1p+\tfrac1q=1 \,,
$$
and set
\begin{equation}\label{eq:def-lambdaC}
\lambda_C := \min \{  \lambda_m : m \in M\}\,.
\end{equation}
(with the usual convention that $\lambda_C=\infty$ if $M$ is empty).
Then, there exists a constant $K_C(u,f,g)$, depending on the norms of $u$ and the data $f$, $g$ in $\Omega_C$ but not on $h_\Gamma$, such that
\begin{equation}\label{eq:bound-Rh-nearC-6}
\Vert h^{-1/2} R_h u \Vert_{0,\tilde{\Gamma}_{h,C}} \leq K_C(u, f,g)  \, h_{\Gamma}^{\min(3/2 - 1/p,\lambda_C) +\zeta' }
\end{equation}
with $\zeta' =\min(2-1/p,\lambda_C)\, \zeta$.
\end{prop}

\begin{rem}\label{rem:logr-case}
{\rm 
If $\frac{2\omega}{q\pi}$ is an integer, $r^\lambda$ should be replaced by $r^\lambda \log r$ in \eqref{eq:2d-singular}. Carrying on a similar analysis, it is easily seen that the logarithmic term can be absorbed by taking any $\zeta' < \min(2-1/p,\lambda_C)\, \zeta$ in \eqref{eq:bound-Rh-nearC-6}. For the sake of simplicity, we will not consider this exceptional case further on.
}
\end{rem}

\subsection{Analysis near an edge or a vertex in a polyhedron} \label{sec:remainder-3D}

As in the two-dimensional case, the solution $u$ can be split into a regular part $u_\text{reg}$ and a singular part $u_\text{sing}$, where the latter is a sum of contributions associated with each edge and each vertex of the polyhedron. Any such contribution is locally a linear combination of terms of the form
\begin{equation}\label{eq:3d-singular}
w(r,\sigma) = r^\lambda \varphi(\sigma) 
\end{equation}
(possibly with $r^\lambda$ multiplied by logarithmic terms), 
where $(r,\sigma)$ form a system of cylindrical coordinates around the edge, or a system of polar coordinates around the vertex, $\varphi$ is an eigenfunction of a Laplace or Laplace-Beltrami operator in the variables $\sigma$ with Dirichlet boundary conditions, and $\lambda >0$ is an expression depending upon the corresponding eigenvalue. We refer to \cite{DaugeBook,DaugeNotes,Maz'yaRossmann} for more details; see also \cite{DemlowNotes}. Precisely, near an edge of aperture $\omega$, the strongest singularity of type \eqref{eq:3d-singular} takes exactly the form \eqref{eq:2d-singular} with $\lambda=\frac\pi\omega$.  At a vertex, if $K$ is the infinite cone that locally represents $\Omega$ near the vertex and $S$ is the intersection of $K$ with the unit sphere, then the smallest $\lambda$ in \eqref{eq:3d-singular} is the positive root of a second-degree algebraic equation whose right-hand side is the smallest eigenvalue of the Laplace-Beltrami operator on $S$ with Dirichlet conditions on $\partial S$.

In order to estimate $R_h w$, one can adapt the two-dimensional arguments. In particular, one may distinguish between the portion of $\tG$ which is mapped to the vertex, the portion which is mapped to the edge, and the part which is mapped to a face near the singularity by the closest-point projection. In all cases, one ends up with an estimate analogous to \eqref{eq:bound-Rh-corner} or \eqref{eq:bound-Rh-nearC-4}, i.e., 
\begin{equation}\label{eq:bound-Rh-3D-1}
\Vert h^{-1/2}R_h w \Vert_{0, \tilde{\Gamma}_{h,U}} \leq K \, h_\Gamma^{\lambda(1+\zeta)} \;,
\end{equation}
where $\tilde{\Gamma}_{h,U}$ denotes the portion of $\tG$ near the vertex or the edge. Correspondingly, one obtains 
a result similar to that of Proposition \ref{prop:bound-Rh-C}, namely
\begin{equation}\label{eq:bound-Rh-3D-2}
\Vert h^{-1/2} R_h u \Vert_{0,\tilde{\Gamma}_{h,U}} \leq K_U(f,g)  \, h_{\Gamma}^{\min(3/2 - 1/p,\lambda_U) +\zeta'} \,,
\end{equation}
where $\lambda_U$ is the smallest exponent appearing in \eqref{eq:3d-singular}, or $\infty$ if the edge or vertex singularity is weak enough, and $\zeta' =\min(2-1/p,\lambda_U)\, \zeta$.

\section{Consistency and convergence analysis \label{sec:convergence} }

From now on, we pose the following regularity assumption on the solution of our  Dirichlet problem.
\begin{assumption}\label{ass:regularity}
The solution $u$ of Problem \eqref{eq:Dirichlet} belongs to $W^{1+s,p}(\Omega)$ for some $s,p$ satisfying the conditions stated in Definition \ref{def:Wh}.
\end{assumption}

A few comments on this assumption are in order. If the boundary $\G$  is $C^2$ and the data satisfy $f \in L^p(\Omega)$ and $g \in W^{2-1/p,p}(\G)$, then $u \in W^{2,p}(\Omega)$, i.e., the assumption is satisfied with $s=1$. On the other hand, if $\G$ contains a finite number of corners (in two dimensions) or edges or vertices (in three dimensions), then $u$ can be split into a regular part $u_\text{reg} \in W^{2,p}(\Omega)$ and a (possibly) singular part $u_\text{sing}$. The latter is a linear combination of terms of the form $u_i = r_i^{\lambda_i}\varphi_i$ (for $i$ ranging in some finite set ${\mathcal I}$), where the local polar coordinate $r_i$ is the distance from the singularity, $\lambda_i >0$, and $\varphi_i$ is a smooth function depending on the remaining polar/cylindrical coordinate(s). Let ${\mathcal I}_1\subseteq {\mathcal I}$ be the set of indices associated with a corner in two dimensions or an edge in three dimensions; let ${\mathcal I}_2 = {\mathcal I} \setminus {\mathcal I}_1$ be the set of indices associated with a vertex in three dimensions. Then, $u_i \in W^{t,q}(\Omega)$ for any $t,q$ satisfying $t < \lambda_i +2/q$ if $i \in {\mathcal I}_1$, or $t < \lambda_i +3/q$ if $i \in {\mathcal I}_2$. In particular, let us choose $q=2$ and let us set
$$
S_\star := \{ \eta < 1 : \eta=\lambda_i \text{ for some } i \in {\mathcal I}_1, \text{ or } \eta=\lambda_i + 1/2 \text{ for some } i \in {\mathcal I}_2 \}\,.
$$
If $S_\star$ is empty, then $u_\text{sing} \in H^2(\Omega)$; otherwise, $u_\text{sing} \in H^{1+s}(\Omega)$ for any $s < s_\star :=\min S_\star$.  In all cases, Assumption \ref{ass:regularity} is fulfilled.

\medskip
As a consequence of Assumption \ref{ass:regularity}, $u$ belongs to the related space $V(\tO;\tT,s,p)$ introduced in \eqref{eq:def-Wh1}, that we will simply indicate by $V(\tO;\tT)$ whenever no confusion may arise.

In order to estimate the discretization error $u-u_h$ in the `energy' norm, we invoke Strang's Second Lemma, which reads
\begin{equation}\label{eq:strang-1}
\Vert u - u_h \Vert_a \leq \left(1+\alpha^{-1}A \right) E_{a,h}(u) + \alpha^{-1} E_{c,h}(u) \,,
\end{equation}
where $E_{a,h}(u) $ is the approximation error
\begin{equation}\label{eq:strang-2}
E_{a,h}(u) := \inf_{v_h \in V_h} \Vert u - v_h \Vert_{V(\tO;\tT)} \,,
\end{equation}
whereas $E_{c,h}(u)$ is the consistency error
\begin{equation}\label{eq:strang-3}
E_{c,h}(u)  := \sup_{v_h \in V_h} \frac{a_h(u,v_h) - \ell_h(v_h)}{\Vert v_h \Vert_a}\,.
\end{equation}
%{\color{violet} NA: I was wondering if it would be better to use Strang's Second Lemma to estimate $\Vert u - u_h \Vert$ in the extended  $V(\tO;\tT)$-norm rather than the $a$-norm since in the duality argument the  $\Vert \psi \Vert$ is given in the extended $V(\tO;\tT)$-norm.}
We are going to estimate both errors in terms of the meshsize $h_\Omega$ introduced in \eqref{eq:mesh-param}.

\begin{prop}\label{prop:estim-approx-err}
There exists a constant $C_a>0$ independent of $u$ and the meshsize such that
\begin{equation}\label{eq:estim-approx-err}
E_{a,h}(u) \leq C_a h^s_\Omega \, |\nabla u |_{s,p,\tO}\,.
\end{equation}
\end{prop}
\begin{proof}
Thanks to the assumptions on $s,p$  in Definition \ref{def:Wh}, $W^{1+s,p}(\Omega) \subset {\mathcal C}^0(\bar{\Omega})$ with continuous injection. Hence, the piecewise linear interpolant $I_h u$ at the nodes of $\tT$ is well-defined in $V_h$. So, we will estimate $\Vert u - I_h u \Vert_{V(\tO;\tT)}$. To this end, we observe that on the reference element $\hat{T}$ there exists a constant $\hat{c}>0$ such that 
$$
\Vert \, \hat{u} - \hat{I}\hat{u}\, \Vert_{1,\hat{T}} \leq \hat{c} \, |\hat{\nabla} \hat{u} |_{s,p,\hat{T}} \qquad \forall \hat{u} \in W^{1+s,p}(\hat{T})\,,
$$
where $\hat{I}$ denotes linear interpolation at the vertices of $\hat{T}$. Using the standard transformations from an element $T \in \tT$ to the reference element and back, one gets
\begin{equation}\label{eq:H1error-bound}
 \Vert h^{-1} (u-I_h u) \Vert_{0,\tO} + \Vert \nabla(u-I_h u) \Vert_{0,\tO} \leq c \, h^s_\Omega \, |\nabla u |_{s,p,\tO}\,.
\end{equation}
Here and in the sequel, $c$ denotes a positive constant independent of the meshsize, which may be different in different formulas. On the other hand,
$$
\Vert h^{-1/2} \tS (u-I_h u) \Vert_{0,\tG} \leq \Vert h^{-1/2} (u-I_h u) \Vert_{0,\tG} + \Vert h^{-1/2} \Vert \bs{d} \Vert \nabla(u-I_h u) \Vert_{0,\tG}\,. 
$$
Using \eqref{eq:scaled-3} with $p=2$ and $s=\sigma=1$ for bounding the first norm on the right-hand side, and \eqref{eq:d-smallness} with \eqref{eq:scaled-5} for bounding  the second norm, one gets
\begin{eqnarray*}
\Vert h^{-1/2} \tS (u-I_h u) \Vert_{0,\tG} &\leq& c \left(\Vert h^{-1} (u-I_h u) \Vert_{0,\tO^b} + (1+h^\zeta_\Gamma)\Vert \nabla(u-I_h u) \Vert_{0,\tO^b} \right. \\
&& \hskip 3.6cm   \left. + \ h^{s_p+\zeta}_\Gamma \, |\nabla (u-I_hu) |_{s,p,\tO^b} \right)\,.
\end{eqnarray*}
We conclude by \eqref{eq:H1error-bound}, after observing that $s_p \geq s$ and $|\nabla (u-I_hu) |_{s,p,\tO} = |\nabla u |_{s,p,\tO}$.  \qquad \end{proof}

Next, we estimate the consistency error $E_{c,h}(u) $. This will be accomplished in two steps.

\begin{lemma}\label{lem:consist-1}
There exists a constant $C_{c,1}>0$ independent of $u$ and the meshsize such that
\begin{equation}\label{eq:est-consist-1}
E_{c,h}(u)  \leq C_{c,1} \Vert h^{-1/2} R_h u \Vert_{0,\tG} \,.
\end{equation}

\end{lemma}
\begin{proof}
By integration by parts and application of \eqref{eq:scaled-2}, one gets for all $v_h \in V_h$
\begin{eqnarray*}
&& \hskip -0.8cm \left | a_h(u,v_h) -\ell_h(v_h) \right| \leq \left | (\tS u - \bar{g}, \partial_\tns v_h)_{0,\tG} \right| + \gamma \, \left|(h^{-1} (\tS u-\bar{g}), \tS v_h)_{0,\tG} \right| \\
&& \hskip 2.23cm \leq  \Vert h^{-1/2} (\tS u-\bar{g} ) \Vert_{0,\tG} \!\! \left( C_I^{1/2} \Vert \nabla v_h \Vert_{0,\tO} \!\! +  \gamma \, \Vert h^{-1/2} \tS v_h\Vert_{0,\tG} \! \right),
\end{eqnarray*}
whence the result. \end{proof}

The analysis of the Taylor remainder carried out in Sect. \ref{sec:remainder} provides an estimate for the scaled norm of $R_h u$ on $\tG$. To get the result, from now on we assume that $\Gamma$ is $C^2$, except possibly for a finite number of singularities (corners in two dimensions, edges and vertices in three dimensions).
Obviously, if $\Gamma$ contains singularities, only those facing the region where $\Omega$ differs from $\tO$ affect the result. To be precise, with reference to the discussion following Assumption \ref{ass:regularity}, let us define ${\mathcal I}_S \subset {\mathcal I}$ as the set of all indices associated with those singularities  that belong to  $\G \setminus \tG$ or to its closure, and let us set
\begin{equation}\label{eq:min-lambda}
\lambda_S := \min \{ \lambda_i : i \in {\mathcal I}_S  \text{ and } \lambda_i < 1 \}\,.
\end{equation}
We suppose that the assumptions made in Proposition \ref{prop:bound-Rh-S} for all the smooth parts of $\Gamma$, and in Proposition \ref{prop:bound-Rh-C} for the part of $\Gamma$ around any corner (or in their three-dimensional counterparts) hold true. Then, putting together the local estimates of the norm of $R_h u$ as in \eqref{eq:bound-Rh-nearC-6} or \eqref{eq:bound-Rh-3D-2}, we obtain the following result.

\begin{lemma}\label{lem:consist-2}  
Setting
\begin{equation}\label{eq:def-exponent-sigma}
\sigma := \min(3/2-1/p,\lambda_S)+\min(2-1/p,\lambda_S)\zeta \;,
\end{equation}
there exists a constant $C_{c,2}>0$ independent of $u$ and the meshsize such that
\begin{equation}\label{eq:est-consist-2}
\Vert h^{-1/2}R_h u \Vert_{0,\tG} \leq C_{c,2} h_\Gamma^\sigma \, K(u,f,g) \;,
\end{equation}
where $K(u,f,g)$ is a sum of norms of $u$ and the data $f$ and $g$.
\end{lemma}

Concatenating \eqref{eq:est-consist-1} and \eqref{eq:est-consist-2}, we arrive at the following estimate of the consistency error.
\begin{prop}\label{prop:estim-consist-err}
Under the previous assumptions, there exists a constant $C_c>0$ independent of $u$ and the meshsize such that
\begin{equation}\label{eq:estim-consist-err}
E_{c,h}(u) \leq C_c \, h_\Gamma^\sigma  \, K(u,f,g)   \,.
\end{equation}
\end{prop}

Summarizing, we obtain an estimate of the discretization error $u-u_h$ in the `energy' norm by using \eqref{eq:estim-approx-err} and \eqref{eq:estim-consist-err}  in \eqref{eq:strang-1}, together with the inequality $h_\Gamma \leq h_\Omega$. 

\begin{thm}\label{theo:estim-discr-err-a} 
Under the assumptions made in this section, setting
\begin{equation}\label{eq:def-exponent-r}
r := \min \left( s , \sigma \right) \,,
\end{equation}
there exists a constant $C>0$ independent of $u$ and the meshsize such that
\begin{equation}\label{eq:estim-discr-err-a}
\Vert u - u_h \Vert_a \leq C \, h_\Omega^r \, ( |\nabla u |_{s,p,\Omega} + K(u,f,g) )\,.
\end{equation}
\end{thm}

\begin{rem}\label{rem:optimality} 
{\rm
This estimate is optimal. Indeed, if data are sufficiently smooth and there is no boundary singularity (which corresponds to setting $\lambda_S=\infty$) in \eqref{eq:def-exponent-sigma}, then $s=1$ and $\sigma \geq 1$ (since $p \geq 2$), whence $r=1$.  On the other hand, if a boundary singularity exists (thus, $\lambda_S <1$), then $s$ is any real number $<\lambda_S$ whereas $\sigma =\lambda_S(1+\zeta)$, whence $r$ is any real number $<\lambda_S$.
}
\end{rem}

\medskip
Finally, observing that  $| h^{s_p} \, \nabla (u-u_h) |_{s,p, \tO^b} = | h^{s_p} \, \nabla u |_{s,p, \tO^b} \leq h_\Omega^s \, | \nabla u |_{s,p, \Omega}$, we immediately obtain a control of the error in the norm of $V(\tO;\tT)$, too.
\begin{corollary}\label{prop:estim-discr-err-V}   
There exists a constant $C>0$ independent of $u$ and the meshsize such that
\begin{equation}\label{eq:estim-discr-err-V}
\Vert u - u_h \Vert_{V(\tO;\tT)} \leq C \, h_\Omega^r \, ( |\nabla u |_{s,p,\Omega} + K(u,f,g) )\,.
\end{equation}
\end{corollary}

\section{Enhanced error estimate in $L^2$ \label{sec:duality} }

In this section, we provide an estimate of the $L^2$-norm of the discretization error $u-u_h$, by adapting to the present setting the classical Aubin-Nitsche duality argument. To this end, we need two technical results. The first one is a measure of the non-symmetry of the bilinear form $a_h$.

\begin{lemma}\label{lemma:non-symmetry}
It holds $\forall v,w \in V(\tO;\tT)$
\begin{equation}\label{eq:non-symmetry}
a_h(w,v)-a_h(v,w) = (\partial_\tns w, \nabla v \cdot \bs{d})_{0,\tG} - (\partial_\tns v, \nabla w \cdot \bs{d})_{0,\tG} \,.
\end{equation}
\end{lemma}
\begin{proof}
By \eqref{eq:SBM-a-form}  one has
\begin{equation*}
\begin{split}
&a_h(w,v) = (\nabla w, \nabla v)_{0,\tO} - (\partial_\tns w, \tS v)_{0,\tG} + (\partial_\tns w, \nabla v \cdot \bs{d})_{0,\tG}  \\
& \hskip 3.8cm  -  ( \tS w, \partial_\tns v)_{0,\tG} + \gamma \, (h^{-1} \tS w, \tS v)_{0,\tG} \,.
\end{split}
\end{equation*}
Subtracting from this expression the analogous one in which $w$ and $v$ are exchanged, we obtain the result.
\end{proof}

The second technical result replaces the Galerkin orthogonality property of standard interior discretizations.

\begin{lemma}\label{lemma:non-Gal-orth}
It holds
\begin{equation}\label{eq:non-Gal-Orth}
a_h(u-u_h, v_h) = (R_h u, \partial_\tns v_h)_{0,\tG} - \gamma (h^{-1}R_h u, \tS v_h)_{0,\tG}\,, \qquad \forall v_h \in V_h \,.
\end{equation}
\end{lemma}
\begin{proof}
Consider \eqref{eq:nitsche1} and replace therein the expression of $u-\bar{g}$ given by \eqref{eq:u-g}, obtaining
\begin{equation*}
\begin{split}
&(\nabla u, \nabla v_h)_{0,\tO} \!\! - (\partial_\tns u, v_h)_{0,\tG} \!\! -  ( \tS u + R_h u, \partial_\tns v_h)_{0,\tG} \!\! +  \gamma \, (h^{-1} \tS u +R_h u, \tS v_h)_{0,\tG} \\
& \hskip 3cm = (f,v_h)_{0,\tO} -  (\bar{g}, \partial_\tns v_h)_{0,\tG} + \gamma \, (h^{-1} \bar{g}, \tS v_h)_{0,\tG}\,, \quad \forall v_h \in V_h \,.
\end{split}
\end{equation*}
Subtracting from this identity the definition \eqref{eq:SBM-formulation-1} of SBM solution, we obtain the result.
\end{proof}

We wish to estimate 
\begin{equation}\label{eq:Riesz}
\Vert u-u_h \Vert_{0,\tO} = \sup_{\tilde{z} \in L^2(\tO)} \frac{(u-u_h,\tilde{z})_{0,\tO}}{\Vert \tilde{z} \Vert_{0,\tO}}\,.
\end{equation}
Given $\tilde{z} \in L^2(\tO)$, let $z \in L^2(\Omega)$ be its extension by $0$ outside $\tO$, and let $w$ be the solution of the auxiliary boundary-value problem
\begin{equation}\label{eq:aux-Dirichlet}
\begin{split}
-\Delta w &= z \quad \text{in \ } \Omega \,, \\
w &=0 \quad \text{on \ } \Gamma \,.
\end{split}
\end{equation}
Let $s_0 \in (\tfrac12,1]$ be such that $w \in H^{1+s_0}(\Omega)$; thus, $s_0=1$ if there are no boundary singularities ($\lambda_S=\infty$ in \eqref{eq:min-lambda}), otherwise it is any real number satisfying $s_0 < \lambda_S<1$. It holds $\Vert w \Vert_{1,\Omega} + |\nabla w|_{s_0,2,\Omega} \leq c \Vert z \Vert_{0, \Omega}= c \Vert \tilde{z} \Vert_{0, \tO}$ for some $c>0$ independent of $\tilde{z}$. Thus, $w \in W(\tO;\tT) := V(\tO;\tT,s_0,2)$ as defined in \eqref{eq:def-Wh1}.

The same argument that led to \eqref{eq:trace-u} shows that the trace $\tilde{q}$ of $w$ on $\tG$ satisfies
\begin{equation}\label{eq:trace-w}
\tilde{q} = -\nabla w \cdot \bs{d} - R(w,\bs{d})\,,
\end{equation}
whence on $\tG$ it holds
\begin{equation}\label{eq:w-q}
0= w-\tilde{q}  \ = \ \tS w  + R_h w \,.
\end{equation}
Thus, for any $v \in V(\tO;\tT)$ we have
\begin{equation*}
\begin{split}
(\tilde{z},v)_{0,\tO} &= -(\Delta w,v )_{0,\tO}  = (\nabla w, \nabla v)_{0,\tO} - (\partial_\tns w, v)_{0,\tG} \\
&= (\nabla w, \nabla v)_{0,\tO} - (\partial_\tns w, v)_{0,\tG} - (\tS w +R_h w, \partial_\tns v)_{0,\tG} \\
& \hskip 4cm + \gamma (h^{-1}(\tS w + R_h w), \tS v)_{0,\tG} \\
&= a_h(w,v) - (R_h w, \partial_\tns v)_{0,\tG} + \gamma (h^{-1} R_h w, \tS v)_{0,\tG} \\
& = a_h(v,w) + (\partial_\tns w, \nabla v \cdot \bs{d})_{0,\tG} - (\partial_\tns v, \nabla w \cdot \bs{d})_{0,\tG} \\ 
& \hskip 4cm - (R_h w, \partial_\tns v)_{0,\tG} + \gamma (h^{-1} R_h w, \tS v)_{0,\tG} \,,
\end{split}
\end{equation*}
where in the last step we have used Lemma \ref{lemma:non-symmetry}. We now pick $v = \psi := u-u_h$ and we use Lemma \ref{lemma:non-Gal-orth} with $v_h = w_I :=I_h w$ to get
\begin{equation}\label{eq:dua1}
\begin{split}
(\tilde{z},\psi)_{0,\tO} &= a_h(\psi, w) + E_1(\psi, w) + E_2(\psi,w)  \\
&= a_h(\psi, w- w_I) + E_1(\psi, w) + E_2(\psi,w) + E_3(u,w_I) \,, 
\end{split}
\end{equation}
with
\begin{equation}\label{eq:def-E1}
E_1(\psi, w) := (\partial_\tns w, \nabla \psi \cdot \bs{d})_{0,\tG} - (\partial_\tns \psi, \nabla w \cdot \bs{d})_{0,\tG} \,,
\end{equation}
\begin{equation}\label{eq:def-E2}
E_2(\psi,w) := - (R_h w, \partial_\tns \psi)_{0,\tG} + \gamma (h^{-1} R_h w, \tS \psi)_{0,\tG}\,,
\end{equation}
\begin{equation}\label{eq:def-E3}
E_3(u,w_I) := (R_h u, \partial_\tns w_I)_{0,\tG} - \gamma (h^{-1}R_h u, \tS w_I)_{0,\tG}\,.
\end{equation}

We proceed to bound the four error terms on the right-hand side of \eqref{eq:dua1}. Concerning the first one, we observe that an argument similar to the one used in the proof of Proposition \ref{prop:estim-approx-err} yields the bound
\begin{equation}\label{eq:err-interp-w}
\Vert w-w_I \Vert_{W(\tO;\tT)} \leq C \, h_\Omega^{s_0} |\nabla w|_{\sigma_0,2,\Omega}\;.
\end{equation}
On the other hand, proceeding as in the proof of Proposition \ref{prop:continuity-a}, one gets
$$
|\, a_h(\psi,w-w_I)\, | \leq B \, \Vert \psi \Vert_{V(\tO;\tT)} \, \Vert w-w_I \Vert_{W(\tO;\tT)} 
$$
for some constant $B>0$ independent of the meshsize. Hence, by  \eqref{eq:estim-discr-err-V} we deduce the existence of a constant $C_{0,0}>0$ independent of $u$, $\tilde{z}$ and the meshsize, such that
\begin{equation}\label{eq:first-bound-enh}
|\, a_h(\psi,w-w_I)\, | \leq C_{0,0} h_\Omega^{r+s_0} \, ( |\nabla u |_{s,p,\Omega} + K(u,f,g) ) \, \Vert \tilde{z} \Vert_{0,\tO} \,.
\end{equation}
Concerning the error term $E_1$ in \eqref{eq:dua1}, recalling \eqref{eq:d-smallness} we have 
$$
|\,E_1(\psi, w) \, | \leq c_d \, h_\Gamma^{1/2+\zeta} \Vert \nabla w \Vert_{0,\tG} \Vert h^{1/2}\nabla \psi \Vert_{0,\tG} 
$$
with $\Vert \nabla w \Vert_{0,\tG} \leq c (\Vert w \Vert_{1,\Omega} + |\nabla w|_{\rho,\Omega})$ by the trace theorem, and $\Vert h^{1/2}\nabla \psi \Vert_{0,\tG} \leq \bar{C}_I \, \Vert \psi \Vert_{V(\tO;\tT)}$ by \eqref{eq:scaled-5}. Hence, we deduce the existence of a constant $C_{0,1}>0$ independent of $u$, $\tilde{z}$ and the meshsize, such that
\begin{equation}\label{eq:second-bound-enh}
|\, E_1(\psi, w) \, | \leq C_{0,1} h_\Omega^{r + 1/2+\zeta}  \, ( |\nabla u |_{s,p,\Omega} + K(u,f,g) ) \, \Vert \tilde{z} \Vert_{0,\tO} \,.
\end{equation}
Next, considering the error term $E_2$ in \eqref{eq:dua1}, we have
\begin{equation*}
\begin{split}
|\,E_2(\psi, w) \, | &\leq  \Vert h^{-1/2} R_h w \Vert_{0,\tG} ( \Vert h^{1/2} \nabla \psi \Vert_{0,\tG} + \gamma \Vert h^{-1/2} \tS \psi \Vert_{0,\tG}) \\
& \leq c\, \Vert h^{-1/2} R_h w \Vert_{0,\tG} \Vert \psi \Vert_{V(\tO;\tT)}\,.
\end{split}
\end{equation*}
To estimate the norm involving $R_h w$, we can apply the equivalent statement of Lemma \ref{lem:consist-2} for $w$, with the summation index $p=2$. Thus, 
denoting by $\sigma_0$ the constant defined in \eqref{eq:def-exponent-sigma} with $p=2$, we get
\begin{equation}\label{eq:bound-Rhw}
\Vert h^{-1/2} R_h w \Vert_{0,\tG}  \leq c \, h_\Gamma^{\sigma_0} \, ( |\nabla w |_{s_0,\Omega} + \Vert \tilde{z} \Vert_{0,\tO} )
\leq c' \, h_\Gamma^{\sigma_0} \, \Vert \tilde{z} \Vert_{0,\tO} 
\end{equation}
Hence, we deduce the existence of a constant $C_{0,2}>0$ independent of $u$, $\tilde{z}$ and the meshsize, such that
\begin{equation}\label{eq:third-bound-enh}
|\, E_2(\psi, w) \, | \leq C_{0,2} h_\Omega^{r + \sigma_0}  \, ( |\nabla u |_{s,p,\Omega} + K(u,f,g) ) \, \Vert \tilde{z} \Vert_{0,\tO} \,.
\end{equation}

Finally, the error term $E_3$ in \eqref{eq:dua1} can be equivalently written as
\begin{equation*}
\begin{split}
E_3(u,w_I) & = (R_h u, \partial_\tns w)_{0,\tG} - \gamma (h^{-1}R_h u, \tS w)_{0,\tG} - (R_h u, \partial_\tns (w-w_I))_{0,\tG} \\
& \ \ \ +\gamma (h^{-1}R_h u, \tS (w-w_I))_{0,\tG} \,.
\end{split}
\end{equation*}
Recalling \eqref{eq:w-q}, one has
\begin{equation*}
\begin{split}
|\,E_3(u, w_I) \, |  &\leq  \Vert h^{-1/2} R_h u \Vert_{0,\tG} (h_\Gamma^{1/2} \Vert \nabla w \Vert_{0,\tG} +  \gamma \Vert h^{-1/2} R_h w \Vert_{0,\tG} + \\
& \qquad \qquad  \qquad \qquad + \Vert h^{1/2} \nabla (w-w_I) \Vert_{0,\tG} + \gamma \Vert h^{-1/2} \tS (w-w_I) \Vert_{0,\tG}) \\
& \hskip -1cm \leq \Vert h^{-1/2} R_h u \Vert_{0,\tG} (h_\Gamma^{1/2} \Vert \nabla w \Vert_{0,\tG} \!\! +  \gamma \Vert h^{-1/2} R_h w \Vert_{0,\tG} \!\! + c\, \Vert w - w_I \Vert_{W(\tO;\tT)}) \,.
\end{split}
\end{equation*}
Using \eqref{eq:est-consist-2}, \eqref{eq:bound-Rhw} and \eqref{eq:err-interp-w}, and observing that $\min(1/2, \sigma_0,s_0)=1/2$, we deduce the existence of a constant $C_{0,3}>0$ independent of $u$, $\tilde{z}$ and the meshsize, such that
\begin{equation}\label{eq:fourth-bound-enh}
|\, E_3(\psi, w) \, | \leq C_{0,3} h_\Omega^{\sigma+1/2 }  \, K(u,f,g)  \, \Vert \tilde{z} \Vert_{0,\tO} \,.
\end{equation}

We are ready to state our estimate on the $L^2$-norm of the discretization error $u-u_h$. Using the Riesz representation \eqref{eq:Riesz} together with \eqref{eq:dua1} and the bounds \eqref{eq:first-bound-enh}, \eqref{eq:second-bound-enh}, \eqref{eq:third-bound-enh}, and \eqref{eq:fourth-bound-enh}, we arrive at the following result.

\begin{thm}\label{theo:estim-L2-err}
Under the same assumptions for the validity of Theorem \ref{theo:estim-discr-err-a}, setting
$$
r_0:= \min(\, s_0, \, 1/2+\zeta, \, \sigma_0, \, \sigma-r+1/2 \,) \;,
$$
there exists a constant $C>0$ independent of $u$ and the meshsize such that
\begin{equation}\label{eq:estim-discr-err}
\Vert u - u_h \Vert_{0,\tO} \leq C \, h_\Omega^{r+r_0}\, ( |\nabla u |_{s,p,\Omega} + K(u,f,g) )\,.
\end{equation}
\end{thm}

\begin{rem}\label{rem:rate-L2}
{\rm
It is worth analyzing the predicted decay rate in different situations, making the exponent $r+r_0$ explicit. 

In the most favorable case (smooth data and smooth boundary, hence, $\lambda_S=\infty$),
one has $r=s_0=1$, and $\sigma=3/2+2\zeta$, $\sigma_0=1+3\zeta/2$, whence $r+r_0= 3/2+\zeta$. 

On the other hand, in the presence of boundary singularities (i.e., $1/2 < \lambda_S <1$), then $r$ and $s_0$ are any real numbers $< \lambda_S$, whereas $\sigma=\sigma_0=\lambda_S(1+\zeta)$, whence $r+r_0$ is any real number $< \lambda_S+\min(\lambda_S,1/2+\lambda_S\zeta)$.

Note that in both cases, the exponent $r+r_0$ is optimal ($=2$ or $<2\lambda_S$) provided $\zeta \geq 1/2$.
}
\end{rem}

\begin{rem}\label{re:non-optimality}
{\rm
Allowing $\zeta >0$ to be arbitrarily small is of paramount importance in complex engineering applications, as it essentially avoids restrictions in the mesh generation near the boundary. 
In this case, the exponent in \eqref{eq:estim-discr-err} is sub-optimal.  Inspecting the behavior of the four error terms appearing in \eqref{eq:dua1} reveals that in the most favorable situation (smooth data and smooth boundary), the terms $a_h(\psi, w- w_I)$, $E_2(\psi,w)$, and  $E_3(u,w_I)$ exhibit the optimal decay $h_\Omega^2$ (see \eqref{eq:first-bound-enh}, \eqref{eq:third-bound-enh}, and \eqref{eq:fourth-bound-enh} with $p=2$). Thus, the responsible for sub-optimality is the term $E_1(\psi,w)$, which reflects the discrepancy between the normal $\bs{n}$ to the exact boundary $\Gamma$ and the normal $\tn$ to the surrogate boundary $\tG$.

To shed further light, assume that $\G$ is smooth and the mapping \eqref{eq:mappingM} is the closest-point projection, which implies $\bs{d}=\Vert \bs{d} \Vert \bs{n}$. Using the splitting $\partial_\tns = (\tn \cdot \bs{n}) \partial_n + (\tn \cdot \bs{t}) \partial_t$, after cancellation of the common term involving the product $\partial_n w \,\partial_n \psi$, one gets
$$
E_1(\psi,w) = \int_{\tG} \Vert \bs{d} \Vert \, (\tn \cdot \bs{t}) (\partial_n \psi \, \partial_t w - \partial_n w \, \partial_t \psi) \,.
$$
For sufficiently refined meshes, when $\tG$ tends to become parallel to $\G$, the factor $\tn \cdot \bs{t}$ may become small, and cancellations due to sign changes from one edge to the other along $\tG$ may occur. Indeed, the optimal rate $h_\Omega^2$ for $\Vert u - u_h \Vert_{0,\tO}$ is often observed in practical calculations with SBM.

Analogous considerations apply in the presence of geometric singularities.
}
\end{rem}

\section{A numerical test \label{sec:numerical_test} }
To validate our theoretical findings, we perform a two-dimensional convergence test for the Galerkin discretization scheme \eqref{eq:SBM-formulation-1} with penalty parameter $\gamma = 10$ in the domain $\Om$ shown in Figure \ref{fig:NumericalTest_NonSmoothBoundary}. The domain $\Omega$ exhibits a re-entrant corner $C$ of amplitude $\omega=\frac{3\pi}2$, formed by two slightly curved edges. The exact solution, in a polar coordinate system $(\varrho, \theta)$ centered at $C$, is given by
\begin{equation} \label{eq:sol-test}
u(\varrho, \theta) = \varrho^{2/3} \, \sin(2\theta/3  + \pi/6 )
\end{equation}

Figure \ref{fig:numerical_solutions_convRates} displays the surrogate domain and boundary for a given mesh, the corresponding numerical solution and the error dacay rates on $u$ and $\nabla u$ in the $L^{2}(\ti{\Om})$-norm. It is apparent from this test involving a re-entrant corner, that the actual convergence rate in  the  energy norm is coherent with our theoretical prediction  given in Theorem \ref{theo:estim-discr-err-a}, as observed in Figure \ref{fig:gradU_convRate}; indeed, our solution $u$ belongs to $H^{1+r}(\Omega)$ for any $r < \tfrac23$.  On the other hand, Theorem \ref{theo:estim-L2-err} predicts a convergence rate in  the $L^2$-norm not smaller than $\tfrac76$, whereas the optimal rate is slightly larger, namely $\tfrac43$;  
Figure \ref{fig:U_convRate} indicates that the latter is probably the actual rate, in accordance with the heuristic considerations given in Remark \ref{re:non-optimality}.
 The actual error values and numerical convergence rates are reported in Table~\ref{table1}.
\begin{figure}
	\centering
\begin{tikzpicture}
\begin{axis}[
axis x line=none,
axis y line=none,
domain=-0.6:0.6,
samples=1001,
disabledatascaling,
xticklabels=\empty,
]
\addplot [black,line width = 0.25mm] {-abs(rad(atan(x)))};
\addplot [black,line width = 0.25mm] {0.55};
\addplot [mark=none,black,line width = 0.25mm] coordinates {(0.6, 0.55) (0.6, -0.54)};
\addplot [mark=none,black,line width = 0.25mm] coordinates {(-0.6, 0.55) (-0.6, -0.54)};
\draw[black,->,-stealth] (0.0707106*0.75,-0.0707106*0.75) arc (-45:225:0.075);
\node[text width=0.2cm] at (0,0.125) { $\omega$};
\dimline[extension start length =0,extension end length = 0,line style = {line width=0.7}]{(-0.6,0.62)}{(0.6,0.62)}{$L$};
\dimline[extension start length =0,extension end length = 0,line style = {line width=0.7}]{(-0.69,-0.54)}{(-0.69,0.55)}{$D$};
\fill (0,0) circle (0.5mm);
\node[text width=0.25cm] at (0,-0.1) { $C$};
\end{axis}
\end{tikzpicture}
	\caption{Domain $\Om$ with $\omega = 3\pi/2$, $C=(0,0)$, $L = 1.2$ and $D = 0.55+ \arctan(0.6)$.}
\label{fig:NumericalTest_NonSmoothBoundary}
\end{figure}
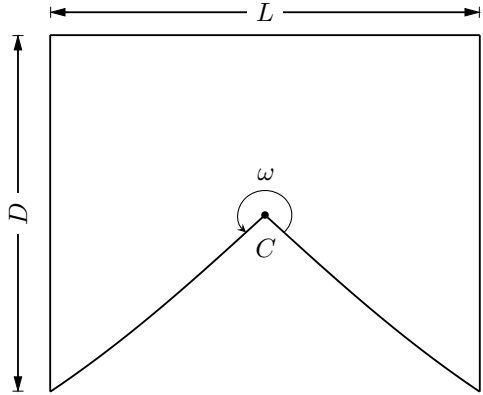
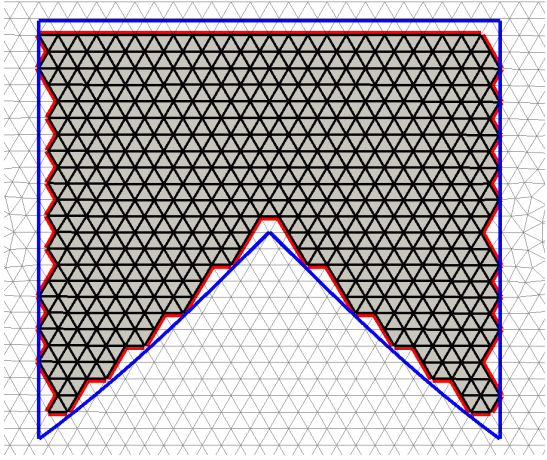
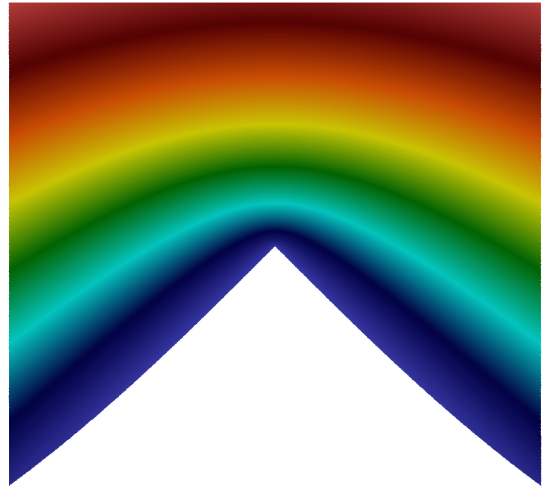
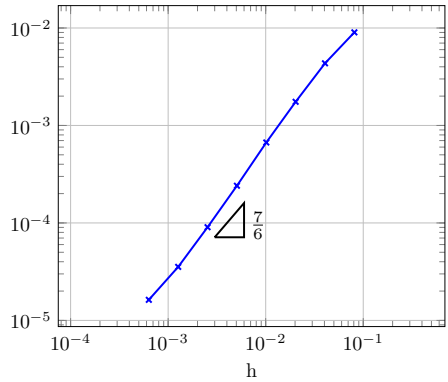
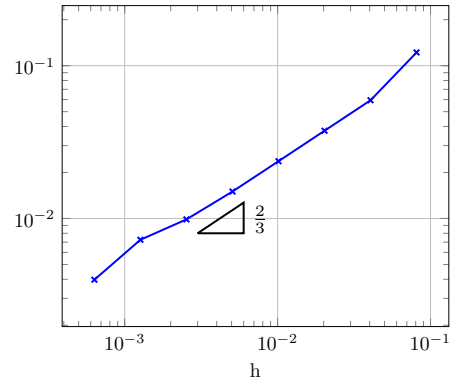
\begin{figure}[t!]\centering	
\begin{subfigure}{0.46\textwidth}\centering
	\includegraphics[width=1\linewidth]{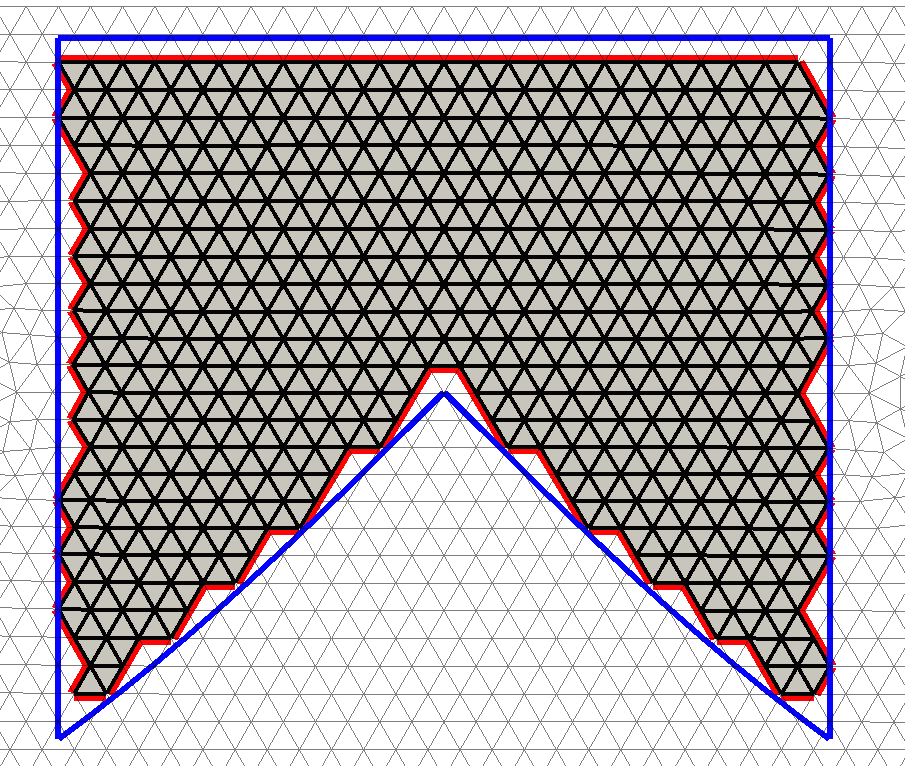}
	\label{fig:efig1}
	\caption{Surrogate domain $\tO$ (grey) with corresponding mesh, true boundary $\G$ (blue) and surrogate boundary $\tG$ (red).}
\end{subfigure}
\qquad
\begin{subfigure}{0.46\textwidth}\centering
	\includegraphics[width=1\linewidth]{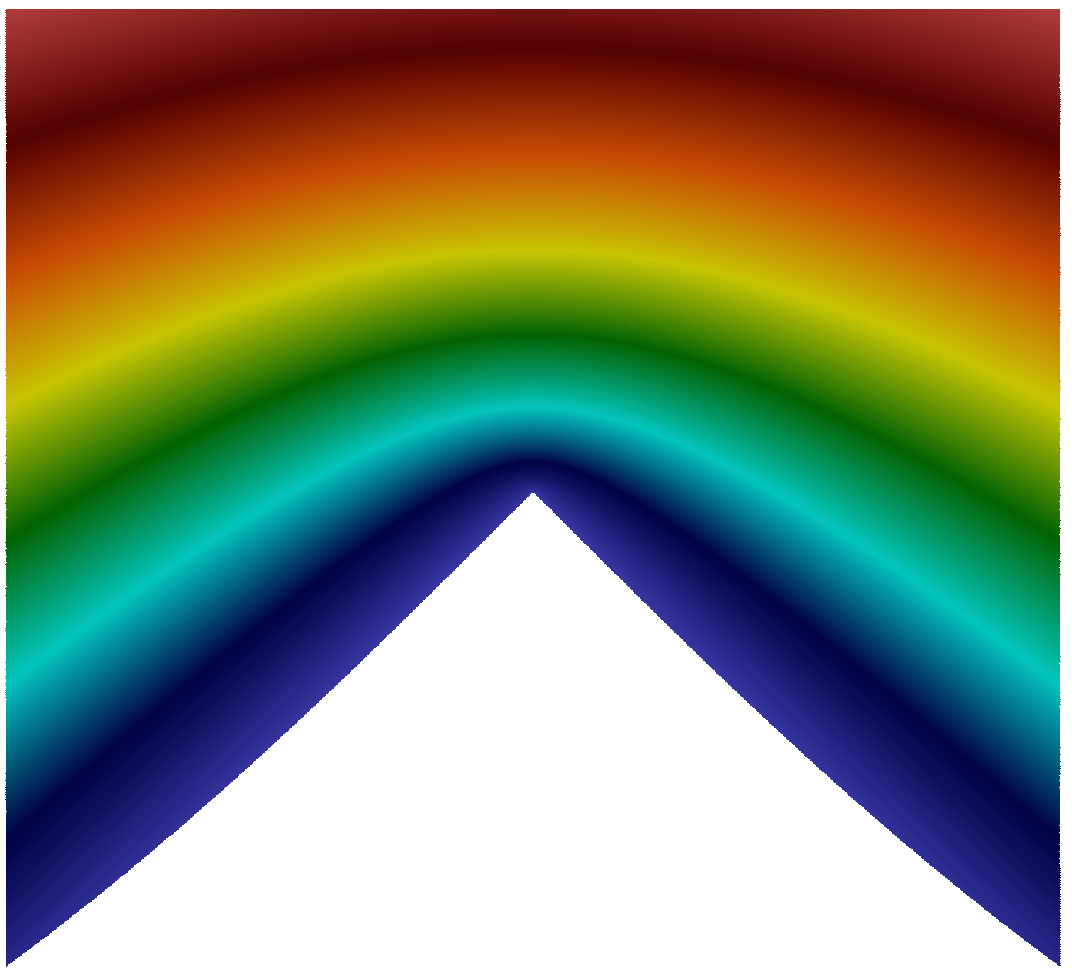}
	\label{fig:efig2}
	\caption{Numerical solution $u_h$.}
\end{subfigure}
\\[.3cm]
\begin{subfigure}[b]{0.46\linewidth}
	\centering
	\begin{tikzpicture}[scale= 0.75]
	\begin{loglogaxis}[
	grid=major,x post scale=1.0,y post scale =1.0,
	xlabel={h},ylabel={},
	legend style={font=\footnotesize,legend cell align=left},
	legend pos= north west,
	axis equal	]
	\addplot[ mark=x,blue,line width=1pt] coordinates {
		(0.000633897,0.0000162644)%mesh4e5
		(0.00126778,0.000035424)%mesh4e5
		(0.0025355,0.0000904343)%mesh4e5
		(0.00507076,0.000240725)%mesh4e5
		(0.010141,0.000669068)%mesh4e5
		(2.027730E-02,0.00174916)%mesh 9e4
		(0.0405394,0.00434224)%mesh 2e4
		(0.0810849,0.00901982)%mesh 2e4
	}; 
	\addplot[black,line width=1pt] coordinates {
	(3E-3,2*0.0000356359487256136)
	(6E-3,0.00016)
	(6E-3,2*0.0000356359487256136)
	(3E-3,2*0.0000356359487256136)
}; 
	\end{loglogaxis}
	\node[text width=0.6cm] at (3.8,1.8) {\footnotesize $\frac{7}{6}$};
	\end{tikzpicture}
	\caption{Convergence rate of $u_h$ to $u$ in $L^{2}$-norm.}
	\label{fig:U_convRate}
\end{subfigure}
\qquad
\begin{subfigure}[b]{0.46\linewidth}
	\centering
	\begin{tikzpicture}[scale=0.75]
	\begin{loglogaxis}[
	grid=major,x post scale=1.0,y post scale =1.0,
	xlabel={h},ylabel={},
	legend style={font=\footnotesize,legend cell align=left},
	legend pos= north west,
	axis equal	]
	\addplot[ mark=x,blue,line width=1pt] coordinates {
		(0.000633897,0.00397341)%mesh4e5
		(0.00126778,0.0072572)%mesh4e5
		(0.0025355,0.00986189)%mesh4e5
		(0.00507076,0.0150105)%mesh4e5
		(0.010141,0.0236991)%mesh4e5
		(2.027730E-02,0.0375139)%mesh 9e4
		(0.0405394,0.0593803)%mesh 2e4
		(0.0810849,0.122044)%mesh 2e4
	}; 
	\addplot[black,line width=1pt] coordinates {
	(3E-03,0.008)
	(6E-3,0.0126992084157456)
	(6E-3,0.008)
	(3E-3,0.008)
}; 
	\end{loglogaxis}
	\node[text width=0.2cm] at (3.5,1.9) {\footnotesize $\frac{2}{3}$};
	\end{tikzpicture}
		\caption{Convergence rate of $\nabla u_h$ to $\nabla u$ in $L^{2}$-norm.}
		\label{fig:gradU_convRate}
\end{subfigure}
	\caption{Surrogate domain $\tO$, numerical solution $u_h$, and convergence rates of the error $u-u_h$ in the $L^{2}$-norm and $H^1$-seminorm.}
\label{fig:numerical_solutions_convRates}
\end{figure}
\begin{table}[hbt!]\centering
	\begin{tabular}{M{1.8cm} | M{2.7cm}  M{1.2cm}  | M{2.7cm}  M{1.2cm} }
		  & \multicolumn{4}{c}{ }  \\ 
		Mesh Size & $ \, \| \, u - u^h \, \|_{0,\tO} $ & Rate & $ \, \| \, \nabla (u - u^h) \, \|_{0,\tO} $ & Rate \\
		\hline
		8.10E-02 & 9.02E-03 & - & 1.22E-01 & -\\
		4.05E-02 & 4.34E-03 & 1.05 & 5.94E-02 & 1.04 \\
		2.03E-02 & 1.75E-03 & 1.31 & 3.75E-02 & 0.66 \\
		1.01E-02 & 6.69E-04 & 1.39 & 2.37E-02 & 0.66 \\
		5.07E-03 & 2.41E-04 & 1.47 & 1.50E-02 & 0.66 \\
		2.54E-03 & 9.04E-05 & 1.41 & 9.86E-03 & 0.61\\
		1.27E-03 & 3.54E-05 & 1.35 & 7.25E-03 & 0.44\\
	    6.34E-04 & 1.63E-05 & 1.12 & 3.97E-03 & 0.87 \\
	\end{tabular}
	\caption{ Convergence rates for the Poisson equation with solution \eqref{eq:sol-test} using the SBM approach. 
		\label{table1} }
\end{table}
\clearpage

% main text ends ...

\section*{Acknowledgments}
The support of the Army Research Office (ARO) under Grant W911NF-18-1-0308 is gratefully acknowledged. CC performed this research in the framework of the Italian MIUR Award ``Dipartimenti di Eccellenza 2018-2022" granted to the Department of Mathematical Sciences, Politecnico di Torino (CUP: E11G18000350001), and with the support of the Italian MIUR PRIN Project 201752HKH8-003. 
He is a member of the Italian INdAM-GNCS research group.
 
%\section*{References}
%\bibliographystyle{model1-num-names.bst}
\bibliographystyle{plain}
\bibliography{./SBM_Poisson}

%\appendix
%\input{./LaTeX_files/SBM_stokes_appendix}

\end{document}